\documentclass[11pt,a4paper,reqno]{amsart}
\usepackage{texmac,enumitem}
\usepackage[foot]{amsaddr}
\usepackage[all]{xy}
\usepackage{amsmath}
\xyoption{web}

\makeatletter
\newcommand*{\doublerightarrow}[2]{\mathrel{
  \settowidth{\@tempdima}{$\scriptstyle#1$}
  \settowidth{\@tempdimb}{$\scriptstyle#2$}
  \ifdim\@tempdimb>\@tempdima \@tempdima=\@tempdimb\fi
  \mathop{\vcenter{
    \offinterlineskip\ialign{\hbox to\dimexpr\@tempdima+1em{##}\cr
    \rightarrowfill\cr\noalign{\kern.5ex}
    \rightarrowfill\cr}}}\limits^{\!#1}_{\!#2}}}

\operators{Aut,Div,Pic,Pid,Id,Cl,ab,Norm,cts,Ext,ia,opp,ie,J,M}
\operators{skew,com,Core,fin,op,id,continuous,G,Ar,res,cor,Irr,Ann,Maxspec}
\calsymbols{c}{O,C,B,A,K,M,D,S,F,X,Y,L,H,P}
\bbsymbols{b}{E,A,P}
\fraksymbols{f}{F,p,A,M,m,r,c}

\def\aa{{\rm a}}
\def\ee{{\rm e}}
\def\ii{{\rm i}}
\def\ia{{\rm ia}}
\def\Zz{{\rm Z}}
\def\rs{{\rm S}}

\def\Kbar{\bar{K}}
\def\ZSG{T}
\def\cR{\hat{G}/\!\!\sim}
\def\ccR{\hat{G}/\sim}
\def\cRd{\hat{G}'/\!\!\sim}
\def\ccRd{\hat{G}'\!/\sim}
\def\rC{{\rm C}}
\def\rS{{\rm S}}

\def\Funcclass{\mathcal{R}}

\def\sqsum{\gamma}

\def\tempfac{t}
\def\quadsub{k}

\def\uB{\mathbf{B}}
\def\us{\mathbf{s}}
\def\uu{\mathbf{u}}
\def\uv{\mathbf{v}}
\def\binfty{\boldsymbol{\infty}}
\def\numdiv{\sigma_0}
\DeclareMathOperator{\comp}{c}

\begin{document}

\title[On class groups of random number fields]
{On class groups of random number fields}
\author{Alex Bartel$^1$}
\address{$^1$School of Mathematics and Statistics, University of Glasgow, University Place,
Glasgow G12 8QQ, United Kingdom}
\author{Hendrik W. Lenstra Jr.$^2$}
\address{$^2$Mathematisch Instituut, Universiteit Leiden, Postbus 9512, 2300 RA Leiden, The Netherlands}
\email{alex.bartel@glasgow.ac.uk, hwl@math.leidenuniv.nl}
\date{\today}

\begin{abstract}
The main aim of the present paper is to disprove the Cohen--Lenstra--Martinet 
heuristics in two different ways and to offer possible corrections. We also 
recast the heuristics in terms of Arakelov class groups, giving an 
explanation for the probability weights appearing in the general form of the 
heuristics. We conclude by proposing a rigorously formulated
Cohen--Lenstra--Martinet conjecture.
%
\end{abstract}
\maketitle
\setcounter{tocdepth}{1}
\tableofcontents

\section{Introduction}\noindent
The Cohen--Lenstra--Martinet heuristics \cite{CL,CM} make predictions on the
distribution of class groups of ``random'' algebraic number fields. In the
present paper, we disprove the predictions in two different
ways, and propose possible adjustments. In addition, we show that the heuristics
can be equivalently formulated in terms of Arakelov class groups of number fields.
This formulation has the merit of conforming to the general expectation that a
random mathematical object is isomorphic to a given object $A$ with a probability
that is inversely proportional to $\#\Aut A$. We end by offering two rigorously
formulated Cohen--Lenstra--Martinet conjectures. In particular, we give two possible
precise definitions of the notion of a ``reasonable function'' occurring in
\cite{CL,CM}. The conjectures can likely be further sharpened and strengthened.
We will be pointing out concrete avenues for further research throughout the paper.

The class group $\Cl_F$ of a number field $F$ is in a natural way a module over
$\Aut F$. Our first disproof of the heuristics proceeds by pointing out
restrictions on the possible module structure that had not been taken into
account. To
explain this, recall that for a
ring $R$, the Grothendieck group $\G(R)$ of the category of finitely generated
$R$-modules has one generator $[L]$ for every finitely generated $R$-module $L$,
and one defining relation $[L]+[N]=[M]$ for every short exact sequence
$0\to L\to M\to N\to 0$ of finitely generated $R$-modules.

The rings $R$ that we will mainly be interested in are certain group rings.
For a set $S$ of prime numbers, we write
$\Z_{(S)}=\{a/b: a$, $b\in \Z$, $b\not\in \bigcup_{p\in S\cup\{0\}}p\Z\}$,
which is a subring of $\Q$. Let now $G$ be a finite abelian group,
and let $S$ be a set of prime numbers not
dividing $\#G$. Then if $R$ is a quotient of the group ring
$\Z_{(S)}[G]$, the torsion subgroup $\G(R)_{\tors}$ of $\G(R)$ may
be identified with a product of class groups of certain cyclotomic rings. For
example if $G$ is cyclic of prime order $p$, then
$\G(\Z_{(S)}[G])_{\tors}$ is naturally isomorphic to the class
group of $\Z_{(S)}[\zeta_p]$, where $\zeta_p$ denotes a primitive $p$-th
root of unity.
For more details the reader may consult Section \ref{sec:imagab}.
Assume now that $G$ has even order, and fix an element $c \in G$ of order $2$. 
The ring that occurs in our first disproof of the heuristics is the ring
$T^- = \Z_{(S)}[G]/(1 + c)$. Every finite $T^-$-module represents a class in
$\G(T^-)_{\tors}$. This applies, in particular, to the module
$T^- \otimes_{\Z[G]} \Cl_F$ if $F$ is a number field whose Galois group is
identified with $G$ such that $c$ acts as complex conjugation. As we will
explain in Section \ref{sec:imagab}, the Cohen--Lenstra--Martinet heuristics
predict that if $F$ runs through all such fields,
the class of $T^- \otimes_{\Z[G]} \Cl_F$ is equidistributed in
$\G(T^-)_{\tors}$. We will argue that this stands in contradiction to the 
following theorem, which we will prove in Section 4 as a consequence of the 
Iwasawa Main Conjecture for abelian fields, as proven by Mazur--Wiles in \cite{MW}.

%
If $F$ is a number field, let $\mu_F$ denote the group
of roots of unity in $F$.

\begin{theorem}\label{thm:introimagabelian}
Let $F/\Q$ be a finite imaginary abelian extension, let $G$ be its Galois group,
let $c\in G$ denote complex conjugation, let $S$ be a set of prime
numbers not dividing $\#G$, and let $\ZSG^-=\Z_{(S)}[G]/(1+c)$. Then we have
$[\ZSG^-\otimes_{\Z[G]}\Cl_F]=[\ZSG^-\otimes_{\Z[G]}\mu_F]$
in $\G(\ZSG^-)$.
\end{theorem}\noindent
Theorem \ref{thm:introimagabelian} is a variant of a result of Greither
\cite[Theorem 5.5]{Greither}, who obtained a stronger conclusion
under some additional hypotheses.

One way of obtaining a contradiction between Theorem \ref{thm:introimagabelian}
and the Cohen--Lenstra--Martinet heuristics is as follows. Suppose that $G$ is cyclic of
order $58$, and $S$ consists of all prime numbers not dividing $58$. Then
$\G(\ZSG^-)_{\tors}$ is the class group of $\Z_{(S)}[\zeta_{29}]$, which is elementary
abelian of order $8$, and in Section \ref{sec:imagab}
we will deduce from Theorem \ref{thm:introimagabelian} that
$[\ZSG^-\otimes_{\Z[G]}\Cl_F]$ is trivial for all but one $F$, contradicting the
equidistribution prediction of the heuristics. In the Cohen--Lenstra--Martinet
conjecture that we propose below, we will remove this obstruction by requiring
the set $S$ to be finite, in which case one has $\G(\ZSG^-)_{\tors}=0$. To
formulate correct heuristics, or even conjectures, when $S$ is not assumed to be
finite remains an important problem, which we will turn to in a future paper.

Our second disproof indicates that enumerating fields by discriminant is 
fundamentally flawed in the context of Cohen--Lenstra--Martinet heuristics. 
The heuristics were initially formulated for number fields of degree $2$. 
Degree $1$ was skipped, since there is only one number field of degree $1$; its
class group is trivial, and does not obey a Cohen--Lenstra--Martinet law.
For a similar reason we shall, in our reformulated conjecture, require that the
class $\cF$ of fields be infinite (see the introduction to Conjecture
\ref{conj:ourconj} below). If $\cF$ were finite but non-empty, it would form a
discrete probability space. In that case the 
distribution of the class groups of the fields in $\cF$ would be influenced by 
the class group of each individual field in $\cF$, and when that is the case
one would generally not expect a Cohen--Lenstra--Martinet distribution.
Also, the fact that $\Q$ has class number $1$ affects other number fields,
simply because they have $\Q$ as a subfield. Indeed, if $F$ is a Galois
number field, $G$ is the Galois group, and $S$ is a set of prime numbers
not dividing $\#G$, then the subgroup of $G$-invariants of the group
$\Z_{(S)} \otimes_{\Z} \Cl_F$ equals $\Z_{(S)} \otimes_{\Z} \Cl_{\Q}$ and is 
therefore trivial, not obeying a Cohen--Lenstra--Martinet law.
Enumerating fields by discriminant
may lead to a difficulty that, imprecisely speaking, is a combination of the two
issues just outlined, namely it may induce a discrete probability distribution
on subfields. We will now give a concrete example of this situation, which
suffices to disprove the Cohen--Lenstra--Martinet heuristics.

For a real number $x$, let $\cC(x)$ be the set of cyclic quartic fields
with discriminant at most $x$ inside a fixed algebraic closure $\bar{\Q}$ of
$\Q$, and for a quadratic number field $k\subset \bar{\Q}$ let
$\cC_k(x) = \{F\in \cC(x) : k\subset F\}$.
A more precise version of the following result will be proven in Section
\ref{sec:quartics} as Theorem~\ref{thm:proportionquad}.
\begin{theorem}\label{thm:quartics}
  For every quadratic field $k$ there exists a real number $p_k$ satisfying
  $0\leq p_k<1$, computable to arbitrary precision, and such that:
  \begin{enumerate}[leftmargin=*,label={\upshape(\roman*)}]
    \item if the discriminant of $k$ is either negative or has at least one
      prime divisor that is congruent to $3\pmod{4}$, then $p_k=0$
      and $\cC_k(x)$ is empty for all real numbers $x$;
    \item otherwise one has
      $$
        \lim_{x\to \infty}\frac{\#\cC_k(x)}{\#\cC(x)}=p_k>0;
      $$
    \item one has $\sum_{k\subset \bar{\Q}}p_k=1$, with the sum running over all
      quadratic number fields $k$.
  \end{enumerate}
\end{theorem}\noindent
Thus, enumerating cyclic quartic fields by discriminant defines a
discrete probability distribution on the quadratic subfields. A concrete
counterexample to the Cohen--Lenstra--Martinet heuristics is obtained
from this as follows. Let
$\cC'(x)\subset \cC(x)$ be the subset of those $F\in \cC(x)$ for which
the class number of the quadratic subfield is not divisible by $3$. Then, as we
will explain in Section \ref{sec:quartics}, the Cohen--Lenstra--Martinet heuristics
predict that the limit
$\lim_{x\to \infty}\#\cC'(x)\big/\#\cC(x)$ exists and that
$\lim_{x\to \infty}\#\cC'(x)\big/\#\cC(x)\approx 0.8402$, where the notation
$a\approx b$ means that $a$
rounds to $b$ with the given precision. However, in Theorem \ref{thm:quarticsbody}
we will deduce from Theorem \ref{thm:quartics} that
the limit $\lim_{x\to \infty}\#\cC'(x)\big/\#\cC(x)$ does indeed exist, and one has
$$
\lim_{x\to \infty}\#\cC'(x)\big/\#\cC(x)\approx 0.9914,
$$
contradicting the heuristics.

We propose to use an order of
enumeration that, by work of Wood \cite{MelanieFair}, does not induce a
discrete probability distribution on subfields when the Galois group is abelian,
and is not expected to do so in the generality of our conjecture.
Enumerating number fields by discriminant has also been observed to pose
problems in the context of other questions in arithmetic statistics, which are
not obviously related to the Cohen--Lenstra--Martinet heuristics, see e.g.
\cite{MelanieFair, Agboola}. It would be interesting to classify all finite groups
$G$ such that enumerating $G$-extensions by discriminant exhibits the same bad
behaviour as described above. More broadly, a type of question that we first heard from
Wood is: which invariants of number fields with a given Galois structure
are ``admissible'' for the purposes of ordering number fields in the
Cohen--Lenstra--Martinet heuristics?

Let us now discuss the shift of perspective towards Arakelov class groups.
Let $F$ be a number field. For the definition of the Arakelov class group
$\Pic^0_F$ of $F$, we refer to \cite{Schoof}. It may be compactly described as
the cokernel of the natural map $\comp(\J_F)\to \comp(\J_F/F^\times)$, where $\J_F$
denotes the id\`ele group of $F$ (see \cite{CF}) and $\comp(X)$ denotes the maximal
compact subgroup of $X$; in particular, $\Pic^0_F$ is a compact abelian group.
We denote the Pontryagin dual of $\Pic^0_F$ by $\Ar_F$. It is an immediate
consequence of \cite[Proposition 2.2]{Schoof} that $\Ar_F$ is a finitely generated
discrete abelian group that fits in a short exact sequence of $\Aut F$-modules
\begin{eqnarray}\label{eq:DualArakelov}
0\to \Hom(\Cl_F,\Q/\Z) \to \Ar_F\to \Hom(\cO_F^\times,\Z)\to 0,
\end{eqnarray}
where $\cO_F$ denotes the ring of integers of $F$. Thus, knowing the torsion
subgroup of $\Ar_F$ is equivalent to knowing $\Cl_F$, and knowing its
torsion-free quotient amounts to knowing $\cO_F^\times$ modulo roots of unity.

The Arakelov class group of a number field is often better behaved than either
the class group or the integral unit group. As an example of this phenomenon,
we mention the following analogue of Theorem \ref{thm:introimagabelian} for
real abelian fields, which we will prove in Section \ref{sec:realab}, also relying
on the results of Mazur--Wiles \cite{MW}.
\begin{theorem}\label{thm:introrealabelian}
Let $F/\Q$ be a finite real abelian extension, let $G$ be its Galois group, let
$S$ be a set of prime numbers not dividing $2\cdot \#G$, and let $\ZSG=\Z_{(S)}[G]$.
Then we have the equality $[\ZSG\otimes_{\Z[G]}\Ar_F] = [\ZSG]-[\Z_{(S)}]$ in
$\G(\ZSG)$, where $[\Z_{(S)}]$ denotes the class of $\Z_{(S)}$ with the trivial
$G$-action.
\end{theorem}\noindent
As we will explain in Section \ref{sec:realab}, this theorem expresses that the
class of $\ZSG\otimes_{\Z[G]}\Ar_F$ in $\G(\ZSG)$ is ``as trivial as it can be'',
given the $\Q[G]$-module structure of $\Q[G]\otimes_{\Z[G]}\Ar_F$. There is no
reason to believe that an analogous result holds for either of the other two
terms in the exact sequence (\ref{eq:DualArakelov}).

We also reinterpret Theorems \ref{thm:introimagabelian} and
\ref{thm:introrealabelian} in terms of the so-called oriented Arakelov class
group, a notion that was introduced by Schoof in \cite{Schoof} and of which we
recall the definition in Section \ref{sec:realab}. That reinterpretation
results in a uniform statement for both theorems -- see Theorem \ref{thm:allabelian} --
and moreover in a statement that might conceivably hold for arbitrary finite
Galois extensions and that should point the way towards relaxing the
assumption in Conjecture \ref{conj:ourconj} that $S$ be finite. We intend to take
up this theme in a forthcoming paper.
%

Above we mentioned the principle that, if a mathematical object is ``randomly''
drawn, a given object appears with a probability that is inversely proportional
to the order of its automorphism group. In the context of Arakelov class groups,
however, the relevant automorphism groups are typically of infinite order.
In \cite{us}, we overcame this obstacle to applying the principle by means of
an algebraic theory, the consequences of which we now explain.

%
%
Let $G$ be a finite group, and let $A$ be a quotient of the group ring $\Q[G]$
by some two-sided ideal. If $p$ is a prime number and $S=\{p\}$, then we write
$\Z_{(p)}$ for $\Z_{(S)}$. We say that a prime number $p$ is \emph{good for $A$}
if there is a direct product decomposition $\Z_{(p)}[G]= J\times J'$, where $J$
is a maximal $\Z_{(p)}$-order in $A$, and the quotient map $\Z_{(p)}[G]\to A$
equals the projection $\Z_{(p)}[G]\to J$ composed with the inclusion $J\to A$.
For example, all prime numbers $p$ not dividing $\#G$ are good for all quotients
of $\Q[G]$. Let $S$
be a set of prime numbers that are good for $A$, and let $R$ denote the
image of $\Z_{(S)}[G]$ in~$A$. Let $\cM$ be a set of finite $R$-modules with
the property that for every finite $R$-module $M'$ there is a unique $M\in \cM$
such that $M\cong M'$, and let $\cP$ be a set of finitely generated projective
$R$-modules such that for every finitely generated projective $R$-module $P'$
there is a unique $P\in \cP$ such that $P\cong P'$. Note that $\cM$ and $\cP$
are countable sets.

Assume for the rest of the introduction that $S$ is finite. As we will show in
Section \ref{sec:commens}, it follows from our hypotheses
that for every finitely generated $A$-module $V$, there is a unique $P_V\in \cP$
such that $A\otimes_R P_V\cong_A V$; and that moreover for every finitely
generated $R$-module $M$ satisfying $A\otimes_R M\cong_A V$, there exists a
unique module $M_0\in \cM$ such that $M\cong P_V\oplus M_0$. For a finitely
generated $A$-module $V$, we define $\cM_V=\{P_V\oplus M_0: M_0\in \cM\}$.


Let $V$ be a finitely generated $A$-module.
In Section \ref{sec:commens}, we will deduce from \cite{us} that there is
a unique discrete probability measure $\bP$ on $\cM_V$
with the property that for any isomorphism $L\oplus E\cong M$ of $R$-modules,
where $L$ and $M$ are in $\cM_V$, and $E$ is finite, we have
$$
\bP(L)=(\Aut M:\Aut L)\cdot\bP(M),
$$
where the inclusion $\Aut L\subset \Aut M$ is the obvious one. This condition
expresses that one can think of $\bP(M)$ as being proportional to $1/\#\Aut M$.

We will now formulate an incomplete version of our conjecture on distributions
of Arakelov class groups, leaving the discussion of the crucial missing detail,
the notion of a ``reasonable'' function, to Section \ref{sec:conj}.
If $f$ is a complex valued function on $\cM_V$, then we define the expected
value of $f$ to be the sum $\bE(f)=\sum_{M\in \cM_V}(f(M)\cdot \bP(M))$ if the
sum converges absolutely. 

Let $G$, $A$, $S$, $R$, and $V$ be as above, and assume that $\sum_{g\in G}g=0$
in $A$. Let $K$ be a number field, let $\Kbar$
be an algebraic closure of $K$, and let $\cF$ be the set of all pairs
$(F,\iota)$, where $F\subset \Kbar$ is a Galois extension of $K$ that contains
no primitive $p$-th root of unity for any prime $p\in S$, and $\iota$ is an
isomorphism between the Galois group of $F/K$ and $G$ that induces an
isomorphism $A\otimes_{\Z[G]} \cO_F^\times\cong V$ of $A$-modules. Assume that
$\cF$ is infinite. For all $(F,\iota)\in \cF$, we will view $\Ar_F$ as a
$G$-module via the isomorphism $\iota$. Let $(F,\iota)\in \cF$ be arbitrary.
It follows from the exact sequence (\ref{eq:DualArakelov}) that we have an
isomorphism $A\otimes_{\Z[G]} \Ar_F\cong \Hom(V,\Q)$ of $A$-modules. Moreover,
every finitely generated $\Q[G]$-module is isomorphic to its $\Q$-linear dual (see e.g.
\cite[\S 10D]{CR}), so we have an isomorphism $\Hom(V,\Q)\cong V$ of $A$-modules.
Thus, $R\otimes_{\Z[G]}\Ar_F$ is isomorphic to a unique element of $\cM_V$.

If $F/K$ is a finite extension, let $c_{F/K}$ be the ideal norm of the product
of the prime ideals of $\cO_K$ that ramify in $F/K$. For a positive real number
$B$, let $\cF_{c\leq B}=\{(F,\iota)\in \cF: c_{F/K}\leq B\}$. If $M$ is a finitely
generated $R$-module satisfying $A\otimes_R M\cong_A V$, and $f$ is a function
defined on $\cM_V$, then we write $f(M)$
for the value of $f$ on the unique element of $\cM_V$ that is isomorphic to~$M$.
\begin{conjecture}\label{conj:ourconj}
Let $f$ be a ``reasonable'' complex valued function on $\cM_V$. Then the limit
$$
\lim_{B\to\infty} \frac{\sum_{(F,\iota)\in \cF_{c\leq B}}f(R\otimes_{\Z[G]}\Ar_F)}{\#\cF_{c\leq B}}
$$
exists, and is equal to $\bE(f)$.
\end{conjecture}\noindent
For any $(F,\iota)\in \cF$, the assumption that $S$ only contain prime
numbers that are good for $A$ implies that the short exact
sequence of $R$-modules that one obtains by applying the functor
$R\otimes_{\Z[G]}\bullet$ to the exact sequence \eqref{eq:DualArakelov}
splits, so that the isomorphism class of the $R$-module $R\otimes_{\Z[G]}\Ar_F$
is determined by those of $R\otimes_{\Z[G]}\Cl_F$ and of $R\otimes_{\Z[G]}\cO_F^\times$.
Moreover, the additional assumption that $S$ be finite implies that the isomorphism
class of $R\otimes_{\Z[G]}\cO_F^\times$ is determined by the isomorphism
class of the $A$-module $A\otimes_{\Z[G]}\cO_F^\times$, and in particular
is constant as $(F,\iota)\in \cF$ varies. Thus, for all $(F,\iota)\in \cF$
the isomorphism class of $R\otimes_{\Z[G]}\Ar_F$
carries precisely the same information as the isomorphism class of $R\otimes_{\Z[G]}\Cl_F$.

In order to allow the reader to compare our conjecture with
the Cohen--Lenstra--Martinet heuristics, we will prove in Section \ref{sec:commens}
that the probability distribution $\bP_V$ on $\cM$ that $\bP$ induces on the set of
$\Z_{(S)}$-torsion submodules of $M\in \cM_V$ satisfies, for all $L_0$, $M_0\in \cM$,
$$
\frac{\bP_V(L_0)}{\bP_V(M_0)}=\frac{\#\Hom(P_V,M_0)\cdot\#\Aut(M_0)}{\#\Hom(P_V,L_0)\cdot\#\Aut(L_0)}.
$$
The distribution $\bP_V$ is, in fact, the probability distribution that is used
in the original Cohen--Lenstra--Martinet heuristics, a version of which we will
give in Section \ref{sec:CL}. In other words, the probability distribution that
we postulate on $\Ar_F$ recovers, and explains, the probability distribution
of Cohen--Lenstra--Martinet on $\Cl_F$. The main differences between our conjecture
and the original heuristic are: the hypothesis that $S$ be finite; and the
changed ordering on $\cF$.

In substance, there is only one piece of theoretical evidence for Conjecture
\ref{conj:ourconj} that we are aware of, which also appears to be the only piece
of evidence for the original conjecture. It is the work of Davenport--Heilbronn
\cite{DH}, as generalised by Datskovsky--Wright \cite{DW}, which implies that
Conjecture \ref{conj:ourconj} holds if $G$ has order $2$, and $f$ is
the function $f(M)=\#(M/3M)$. The theorem actually addresses the average of the
function $f$ when the quadratic extensions are enumerated by ideal norm of the
relative discriminant rather than of the product of ramified primes, but also
shows that the average is unchanged if one restricts to quadratic extensions with
prescribed local behaviour at the primes above $2$ -- see also \cite[Corollary 4]{BV16}
for the special case $K=\Q$. It is not hard to deduce from
this refined statement that Conjecture \ref{conj:ourconj} itself also holds for
this choice of $G$ and $f$.

As for computational evidence, in almost all large scale computational data for
Cohen--Lenstra--Martinet related questions assembled to date that we are
aware of fields are enumerated by the absolute value of their discriminant, as
in the original heuristics. It would be desirable to have a large body of
computational data, not only for Conjecture \ref{conj:ourconj}, but also for
variants with other orders of enumeration that do not suffer from the problem
pointed out in the present paper. 

We conclude the introduction by discussing the notion of a ``reasonable
function''. It already appears in the original
heuristics of Cohen--Lenstra \cite{CL} and Cohen--Martinet \cite{CM}, but has
never been made precise. In order to turn the heuristics into actual
conjectures, with a truth value, we will offer two possible
definitions of this notion in Section \ref{sec:conj}.

Of course, the minimal requirement that a function $f$ has to fulfil in order
to satisfy the conclusion of Conjecture \ref{conj:ourconj} is that the expected
value $\bE(f)$ should exist. It may be tempting to
conjecture that, at least for $\R_{\geq 0}$-valued functions $f$, this minimal
condition is in fact sufficient. However, the following result, which we will
prove as Theorem \ref{thm:Bjorn}, using a construction communicated to us
by Bjorn Poonen, shows that the conjecture is likely false in that generality.

\begin{theorem}\label{thm:introBjorn}
Let $X$ be a countably infinite set, and let $p$ be a discrete
probability measure on $X$. For all $x\in X$, abbreviate $p(\{x\})$ to $p(x)$
and assume that $p(x)>0$. Let $B$ be the subset of $X^{\Z_{\geq 1}}$
consisting of those sequences $(y_i)_{i\in \Z_{\geq 1}}\in X^{\Z_{\geq 1}}$
for which there exists a function $f\colon X\to \R_{\geq 0}$ such that
\begin{enumerate}[leftmargin=*, label={\upshape(\roman*)}]
\item the expected value $\sum_{x\in X}f(x)p(x)$ is finite, but
\item the limit $\lim_{n\to \infty}\frac1n\sum_{i=1}^nf(y_i)$ does not
exist in $\R$.
\end{enumerate}
Then the measure of $B$ with respect to the product measure induced by $p$
on $X^{\Z_{\geq 1}}$ is equal to $1$.
\end{theorem}\noindent
Theorem \ref{thm:introBjorn} suggests that if the sequence
of dual Arakelov class groups really were a random sequence with probability
measure $\bP$, then with probability $1$ there would be a function $f$ 
for which $\bE(f)$ exists but the conclusion of Conjecture \ref{conj:ourconj}
is violated, and such functions $f$ should therefore not be considered ``reasonable''.

One of our two ways of narrowing down the class of reasonable functions is to
require, in Conjecture \ref{conj:suppl1}, that for all $j\in \Z_{\geq 1}$ the
$j$-th moment of $|f|$ be finite. In Question \ref{qn:moments} we raise the
probability-theoretic problem of showing that this additional requirement
suffices to eliminate the difficulty mentioned. In our second proposal, in
Conjecture \ref{conj:suppl2}, we take a completely different, computer science perspective.

\begin{acknowledgements}
We would like to thank: Bjorn Poonen for communicating to us the construction
that we use in the proof of Theorem \ref{thm:Bjorn} and for allowing us to
include it here; Cornelius Greither for helpful correspondence on Galois
module structure in abelian fields; Manjul Bhargava and Henri Johnston for
several helpful conversations; Henri Johnston, Melanie Matchett Wood,
and Bjorn Poonen for many helpful comments and suggestions on the first
draft of the manuscript; and an anonymous referee for useful remarks that
helped us improve the exposition.
Parts of this research were done at the MPIM in Bonn, and parts were done
while the first author was at Warwick University. We would like to thank these
institutions for their hospitality and financial support. The first author
gratefully acknowledges the financial support through EPSRC Fellowship
EP/P019188/1, `Cohen-Lenstra heuristics, Brauer relations, and low-dimensional manifolds'.
\end{acknowledgements}

\section{The Cohen--Lenstra--Martinet heuristics}\label{sec:CL}\noindent
In this section, we give a version of the heuristics of Cohen--Lenstra
and Cohen--Martinet. Our formulation differs from the original one in several
ways, as we will point out, but does not yet incorporate the corrections that
will be necessary in light of the results of Sections \ref{sec:imagab} and
\ref{sec:quartics}.

Let $G$ be a finite group, let $A$ be a quotient of the group ring $\Q[G]$ by
some two-sided ideal that contains $\sum_{g\in G}g$, let $S$ be a set of prime
numbers that are good for $A$, let $R$ be the image of $\Z_{(S)}[G]$ in $A$
under the quotient map, let $P$ be a finitely generated projective $R$-module,
and let $V$ denote the $A$-module $A\otimes_{R}P$. Let $\cM$ be as in the
introduction. If $\uB=(B_{A'})_{A'}$ is a sequence of real numbers indexed by
the simple quotients $A'$ of $A$, then let $\cM_{\leq\uB}$ be the set of all
$M\in \cM$ such that for every simple quotient $A'$ of $A$ we have
$\#(R'\otimes_R M) \leq B_{A'}$, where $R'$ denotes the image of $R$ in $A'$.
We will write $\uB\to \binfty$ to mean $\min_{A'}B_{A'}\to \infty$.

If $f$ is a function on $\cM$, and $M$ is a finite $R$-module, then we
will write $f(M)$ for the value of $f$ on the unique
element of $\cM$ that is isomorphic to $M$. For a finite module $M$,
define $w_P(M)=\frac{1}{\#\Hom(P,M)\cdot\#\Aut M}$.

Let $K$ be a number field, let $\Kbar$
be an algebraic closure of $K$, and let $\cF$ be the set of all pairs
$(F,\iota)$, where $F\subset \Kbar$ is a Galois extension of $K$ that contains
no primitive $p$-th root of unity for any prime $p\in S$, and
$\iota$ is an isomorphism between the Galois group of $F/K$ and $G$ that induces
an isomorphism $A\otimes_{\Z[G]} \cO_F^\times\cong V$ of $A$-modules. If
$(F,\iota)$ is in $\cF$, then the class group $\Cl_F$
is naturally a $G$-module via the isomorphism $\iota$.
Let $\Delta_{F/K}$ denote the ideal norm of the relative discriminant of $F/K$.
For a positive real number $B$, let
$\cF_{\Delta\leq B}=\{(F,\iota)\in \cF: \Delta_{F/K}\leq B\}$.

\begin{heuristic}[Cohen--Lenstra--Martinet heuristics]\label{he:CLM}
With the above notation, assume that $\cF$ is infinite, and let $f$ be a
``reasonable'' $\C$-valued function on $\cM$. Then:
\begin{enumerate}[leftmargin=*, label={\upshape(\alph*)}]
\item the limit
\begin{eqnarray}\label{eq:average}
\lim_{B\to \infty}\frac{\sum_{(F,\iota)\in \cF_{\Delta\leq B}}
f(R\otimes_{\Z[G]}\Cl_F)}{\#\cF_{\Delta\leq B}}
\end{eqnarray}
exists;
\item the limit
\begin{eqnarray}\label{eq:expectation}
\lim_{\uB\to \binfty} \frac{\sum_{M\in \cM_{\leq \uB}}w_P(M)f(M)}{\sum_{M\in \cM_{\leq \uB}}w_P(M)}
\end{eqnarray}
exists;
\item the two limits {\rm (\ref{eq:average})} and {\rm (\ref{eq:expectation})} are equal.
\end{enumerate}
\end{heuristic}\noindent
The notion of a ``reasonable'' function was left undefined in \cite{CL} and
\cite{CM}, and has never been made precise.



Let us briefly highlight some of the differences between Heuristic \ref{he:CLM}
and \cite[Hypoth\`ese 6.6]{CM}.

While Cohen--Martinet assume that the families of fields under consideration
are non-empty, we assume in Heuristic \ref{he:CLM} that $\cF$ is infinite. If
$\cF$ was finite, then Heuristic \ref{he:CLM} would almost certainly be false
for any reasonable interpretation of the word ``reasonable'', so we avoid
dependencies on some form of the inverse Galois problem.


The original heuristics did not include the condition that the fields $F$
should not contain a primitive $p$-th root of unity for any $p\in S$. However,
work by Malle \cite{Malle} suggests that if that condition was omitted, then
for those primes $p$ that do divide
the order of the group of roots of unity of $F$ for a positive proportion of all
$(F,\iota)\in \cF$, the probability measure governing the $p$-primary parts of
the class groups would likely need to be modified.

Considering pairs $(F,\iota)$ consisting of a number field and an isomorphism
between its Galois group and $G$, as in Heuristic \ref{he:CLM}, is one way of
making precise the original formulations of Cohen--Lenstra and Cohen--Martinet, who speak of the
family of extensions of $K$ ``with Galois group $G$''. In the above formulation,
some concrete conjectures of \cite{CM2} become trivially true.
An example of this is \cite[(2)(a)]{CM2}, where it is conjectured that if $\rC_3$
is a cyclic group of order $3$, then among Galois number fields with Galois
group isomorphic to $\rC_3$ and with class number $7$, each of the two isomorphism
classes of non-trivial $\rC_3$-modules of order $7$ appears with probability
$50\%$. We do not know how to make the notion of a family of fields ``with
Galois group $G$'' precise in a natural way that would render this and similar
examples well-defined, but not trivial.

In \cite[Hypoth\`ese 6.6]{CM}, the analogue of the expected value
(\ref{eq:expectation}), which is denoted by $M_u^S(f)$ there, is defined as a
sum over representatives of isomorphism classes of all finite $\cO$-modules,
where $\cO$ is a maximal $\Z$-order of $\Q[G]$ containing $\Z[G]$. This differs
in two ways from Heuristic \ref{he:CLM}.

Firstly, in that version the sums that are analogous to the numerator and
denominator of expression (\ref{eq:expectation}) contain summands corresponding
to modules $M$ whose order is divisible by primes not in $S$, even when $f(M)$
depends only on the isomorphism class of $\Z_{(S)}\otimes_{\Z}M$. It can be
easily deduced from \cite[Proposition 5.6]{CL} that this is a purely cosmetic
difference. In particular, one can show that if $S'$ is a subset of $S$, and
$f(M)$ only depends on the isomorphism class of $\Z_{(S')}\otimes_{\Z_{(S)}}M$,
then Heuristic \ref{he:CLM} holds if and only if it holds with $S$ replaced by $S'$.

Secondly, in \cite[Hypoth\`ese 6.6]{CM} the expected value $M_u^S(f)$
does not depend on the quotient $R$ of $\Z_{(S)}\otimes_{\Z}\cO \cong \Z_{(S)}[G]$,
and that is certainly not what was intended: if the set $S$ is finite and non-empty,
and the function $f$ satisfies $f(M)=0$ for all those $\Z_{(S)}[G]$-modules $M$
that are annihilated by the kernel of the map $\Z_{(S)}[G]\to R$, and $f(M)=1$
for all other $\Z_{(S)}[G]$-modules, then the limit in (\ref{eq:average}) is
zero, while $M_u^S(f)$ is non-zero.

\begin{remark}
There are some curious functions that may be considered reasonable, but for
which the limit (\ref{eq:expectation}) does not exist when $S$ is too large.
Suppose, for example,
that $S$ contains almost all prime numbers, and that $R$ is the product of two
rings $T$ and $T'$, where $T\cong T'\cong\Z_{(S)}$, so that $A$ is a product
of two $\Q$-algebras $C=\Q\otimes_{\Z_{(S)}}T$ and $C'=\Q\otimes_{\Z_{(S)}}T'$,
both isomorphic to $\Q$. Take $P=\{0\}$. Define
$f(M) = 1$ if $\#(T\otimes_R M) > \#(T'\otimes_R M)$, and $f(M)=0$ otherwise.
If $B_C$ is fixed and $B_{C'}$ tends to $\infty$, then the limit of
$$
\frac{\sum_{M\in \cM_{\leq \uB}}w_P(M)f(M)}{\sum_{M\in \cM_{\leq \uB}}w_P(M)}
$$
is $0$; while if $B_{C'}$ is fixed and $B_C$ tends to $\infty$, then that limit
is $1$. From this observation it can be deduced that the limit in
(\ref{eq:expectation}) does not exist. Such examples can be realised in the context of number
fields: let $\rC_2$ be a cyclic group of order $2$ and $\rS_3$ the symmetric group on
a set of $3$ elements, suppose that $F/\Q$ is Galois with Galois group isomorphic
to $\rC_2\times \rS_3$ such that the inertia groups at $\infty$ are subgroups of $\rS_3$ of order $2$,
and let $S$ contain almost all prime numbers. Then $F$ contains two imaginary quadratic
subfields, one that is contained in a subextension that is Galois over $\Q$
with Galois group isomorphic to $\rS_3$, and one that is not,
and the question of how often the order of the $S$-class group of the former
is greater than that of the latter cannot be answered by Heuristic \ref{he:CLM}.
We leave it to the reader to check that this example can indeed be realised in
the framework of Heuristic \ref{he:CLM} with a judicious choice of $A$ and $V$.

If instead $S$ is finite, then all the relevant sums converge absolutely, and the
limit (\ref{eq:expectation}) is well-defined. This adds to our reasons for
demanding in Conjecture \ref{conj:ourconj} that $S$ be finite.
\end{remark}

%
%


\section{Commensurability of automorphism groups}\label{sec:commens}\noindent
In this section, we recall the formalism of commensurability from \cite{us},
and deduce the existence of the probability measure $\bP$ as described in the
introduction.

A {\it group isogeny\/} is a group homomorphism
$f\colon H\to G$ with $\#\ker f<\infty$ and $(G:fH)<\infty$, and its
{\it index\/} $\ii(f)$ is defined to be $(G:fH)/\#\ker f$. For
a ring $R$, an $R${\it-module isogeny\/} is an $R$-module homomorphism that
is an isogeny as a map of additive groups. A {\it ring isogeny\/}
is a ring homomorphism that is an isogeny as a map of additive
groups. The index of an isogeny of one of the latter two types is defined as
the index of the induced group isogeny on the additive groups.

If $X$, $Y$ are objects of a category
$\cC$, then a {\it correspondence\/} from $X$ to $Y$ in $\cC$ is
a triple $c=(W,f,g)$, where $W$ is an object of $\cC$ and
$f\colon W\to X$ and $g\colon W\to Y$ are morphisms in~$\cC$; we
will often write $c\colon X\rightleftharpoons Y$ to indicate a
correspondence. A {\it group commensurability\/} is a
correspondence $c=(W,f,g)$ in the category of groups for which
both $f$ and $g$ are isogenies. For a ring $R$, we define
an {\it $R$-module commensurability} to be a correspondence of $R$-modules that
is a commensurability of additive groups, and a {\it ring commensurability}
is defined analogously. If $c=(W,f,g)$ is a commensurability of groups, or of
rings, or of modules over a ring, then its index is defined as
$\ii(c)=\ii(g)/\ii(f)$.

Let $R$ be a ring, and let $c=\!(N,f,g)\colon$
$L\rightleftharpoons M$ be a correspondence
of $R$-modules. We define the {\it endomorphism ring\/} $\End c$
of $c$ to be the subring $\{(\lambda,\nu,\mu)\in(\End
L)\times(\End N)\times(\End M):\lambda f=f\nu$, $\mu
g=g\nu\}$ of the product ring $(\End L)\times(\End
N)\times(\End M)$.  There are natural ring homomorphisms
$\End c\to\End L$ and $\End c\to\End M$ sending
$(\lambda,\nu,\mu)$ to $\lambda$ and $\mu$, respectively;
we shall write $\ee(c)\colon\End L\rightleftharpoons\End M$ for
the ring correspondence consisting of $\End c$ and those two ring
homomorphisms. Similarly, we define the
{\it automorphism group\/} $\Aut c$ of $c$ to be the group
$(\End c)^\times$, and we write $\aa(c)\colon\Aut
L\rightleftharpoons\Aut M$ for the group correspondence
consisting of $\Aut c$ and the natural maps $\Aut c\to\Aut L$,
$\Aut c\to\Aut M$.

The following result is a special case of \cite[Theorem 1.2]{us}.
\begin{theorem}\label{thm:fromus}
Let $G$ be a finite group, let $A$ be a quotient of $\Q[G]$ by some two-sided
ideal, let $S$ be a set of prime numbers, let $R$ be the image of
$\Z_{(S)}[G]$ in $A$, and let $L$, $M$ be finitely generated $R$-modules. Then:
\begin{enumerate}[leftmargin=*, label={\upshape(\alph*)}]
\item\label{item:existence} there is an $R$-module commensurability
$L\rightleftharpoons M$ if and only if the $A$-modules
$A\otimes_{R}L$ and $A\otimes_{R}M$ are isomorphic;
\item if $c\colon L\rightleftharpoons M$ is an $R$-module
commensurability, then $\ee(c)\colon\End L\rightleftharpoons\End M$
is a ring commensurability, and $\aa(c)\colon \Aut
L\rightleftharpoons\Aut M$ is a group commensurability;
\item if $c$, $c'\colon L\rightleftharpoons M$ are
$R$-module commensurabilities, then one has
$$\ii(\ee(c))=\ii(\ee(c')),\;\;\;\;\; \ii(\aa(c))=\ii(\aa(c')).$$
\end{enumerate}
\end{theorem}\noindent
\begin{notation}\label{not:GASM}
For the rest of the section, let $G$, $A$, $S$, and $R$ be as in Theorem \ref{thm:fromus}.
If $L$ and $M$ are finitely generated $R$-modules such that there exists a commensurability
$L\rightleftharpoons M$, then we define $\ia(L,M)=\ii(\aa(c))$ for any commensurability
$c\colon L\rightleftharpoons M$.
\end{notation}

\begin{lemma}\label{lem:multiplicative}
Let $L$, $M$, and $N$ be finitely generated $R$-modules such that there are
commensurabilities $L\rightleftharpoons M$ and $M\rightleftharpoons N$. Then
there is a commensurability $L\rightleftharpoons N$, and we have
$\ia(L,M)\ia(M,N)=\ia(L,N)$.
\end{lemma}
\begin{proof}
See \cite[Theorem 7.3]{us}.
\end{proof}\noindent
\begin{proposition}\label{prop:plusproj}
Let $P$ be a finitely generated projective $R$-module, let $L_0$ and $M_0$ be
finite $R$-modules, let $L= P\oplus L_0$, and $M= P\oplus M_0$. Then there is an $R$-module
commensurability $L\rightleftharpoons M$, and we have
$$
\ia(L,M) = \frac{\#\Hom(P,M_0)\cdot \#\Aut M_0}{\#\Hom(P,L_0)\cdot \#\Aut L_0}.
$$
\end{proposition}
\begin{proof}
By Theorem \ref{thm:fromus}\ref{item:existence}, there exist commensurabilities
$L\rightleftharpoons M$, $P\rightleftharpoons L$, and $P\rightleftharpoons M$.
We will first compute $\ia(P,L)$ and $\ia(P,M)$.
The split exact sequence
$$
0\rar L_{0}\rar L\stackrel{\pi}{\rar} P\rar 0,
$$
where $\pi$ is the natural projection map, induces a surjective map
$$
\Aut L\rar \Aut L_{0}\times\Aut P,
$$
whose kernel is easily seen to be canonically isomorphic to $\Hom(P, L_{0})$.
It follows that if
$c$ is the commensurability $(L,\pi,\id)\colon P\rightleftharpoons L$, then
the map $\Aut c\rar \Aut L$ is an isomorphism, while
the map $\Aut c \rar \Aut P$ is onto, with kernel of cardinality
$\#\Hom(P, L_{0})\cdot \#\Aut L_{0}$.
Hence $\ia(P,L)=\ii(\aa(c)) = \#\Hom(P, L_{0})\cdot \#\Aut L_{0}$,
and analogously for $\ia(P,M)$.
By Lemma \ref{lem:multiplicative}, we therefore have
that
$$
\ia(L,M)=\frac{\ia(P,M)}{\ia(P,L)}
= \frac{\#\Hom(P,M_0)\cdot \#\Aut M_0}{\#\Hom(P,L_0)\cdot \#\Aut L_0},
$$
as claimed.
\end{proof}\noindent
For the rest of the section, assume that $S$ is a finite set of prime numbers
that are good for $A$, and let $\cM$ and $\cP$ be sets of finitely generated
$R$-modules as in the introduction.
\begin{lemma}\label{lem:uniqueproj}
Let $V$ be a finitely generated $A$-module. Then there exists a unique
$P_V\in \cP$ such that $A\otimes_R P_V\cong_A V$. Moreover, if $M$ is a
finitely generated $R$-module such that $A\otimes_R M\cong_A V$, then there
exists a unique $M_0\in \cM$ such that $M\cong P_V\oplus M_0$.
\end{lemma}
\begin{proof}
Let $T$ be any finitely generated subgroup of $V$ such that
$\Q T=V$. Then $P=RT$ is a finitely generated $R$-submodule of $V$ such that
$A\otimes_RP=V$. By \cite[Corollary (11.2)]{MaxOrders}, the ring $R$ is a
maximal $\Z_{(S)}$-order in $A$. It follows from \cite[Theorems (18.1) and
(2.44)]{MaxOrders} that $P$ is a projective $R$-module, and is isomorphic to
a unique element $P_V$ of $\cP$.

Let $M$ be a finitely generated $R$-module such that $A\otimes_R M\cong_A V$,
let $M_{\tors}$ be the $R$-submodule of $M$ consisting of $\Z_{(S)}$-torsion
elements, and let $\bar{M}=M/M_{\tors}$ be the $\Z_{(S)}$-torsion free quotient.
It follows from \cite[Theorem (18.10) and \S18 Exercise 3]{MaxOrders} that
$\bar{M}\cong_R P_V$. Since $P_V$ is projective, we have $M\cong_R P_V\oplus M_{\tors}$,
and $M_{\tors}$ is isomorphic to a unique element $M_0$ of $\cM$.
\end{proof}\noindent
%
%
Recall from the introduction that if $V$ is a finitely generated $A$-module,
and $P_V\in \cP$ is such that $A\otimes_R P_V\cong_A V$, then
we define $\cM_V=\{P_V\oplus M_0: M_0\in \cM\}$.
By Lemma \ref{lem:uniqueproj}, every finitely generated $R$-module $M$ satisfying
$A\otimes_R M\cong_A V$ is isomorphic to a unique element of $\cM_V$.
We now state and prove the main result of the section.
\begin{proposition}\label{prop:bodyequivprob}
Under the assumptions of Notation \ref{not:GASM}, suppose that $S$ is a finite
set of prime numbers that are good for $A$, let $V$ be a finitely generated $A$-module,
and let $P_V\in \cP$ be such that $A\otimes_R P_V\cong V$. Then:
\begin{enumerate}[leftmargin=*, label={\upshape(\alph*)}]
\item\label{item:P_V} there exists a unique discrete probability distribution
$\bP_V$ on $\cM$ such that for all $L_0$, $M_0\in \cM$ we have
$$
\frac{\bP_V(L_0)}{\bP_V(M_0)}=\frac{\#\Hom(P_V,M_0)\cdot\#\Aut(M_0)}{\#\Hom(P_V,L_0)\cdot\#\Aut(L_0)};
$$
\item\label{item:P} there exists a unique discrete probability distribution
$\bP$ on $\cM_{V}$ such that for any isomorphism $L\oplus E\cong M$ of
$R$-modules, where $L$ and $M$ are in $\cM_V$, and $E$ is finite, we have
$$
\bP(L)=(\Aut M:\Aut L)\cdot\bP(M),
$$
where the inclusion $\Aut L\subset \Aut M$ is the obvious one;
\item\label{item:equivdistr} for all $E \in \cM$, we have
$\bP(P_V\oplus E)=\bP_V(E)$.
\end{enumerate}
\end{proposition}
\begin{proof}
For $L_0\in \cM$, write
$$
w(L_0)=\frac{1}{\#\Hom(P_V,L_0)\cdot\#\Aut(L_0)}.
$$
By \cite[Th\'eor\`eme 3.6]{CM}, the sum
$
\sum_{L_0\in \cM}w(L_0)
$
converges, to $\alpha$, say, so that we may define the probability
distribution $\bP_V$ on $\cM$ by $\bP_V(L_0)=w(L_0)/\alpha$ for
$L_0\in \cM$. It satisfies the conclusion of part \ref{item:P_V},
and is clearly the unique such distribution. This proves part \ref{item:P_V}.

We now prove part \ref{item:P}.
By combining the convergence of $\sum_{L_0\in \cM}w(L_0)$ with 
Proposition \ref{prop:plusproj}, we see that for all $M\in \cM_V$, the sum
$\sum_{L\in \cM_V}\ia(L,M)$ also converges, to $\beta_M$, say, so that we may
define a probability distribution $\bP$ on
$\cM_V$ by $\bP(L)=\ia(L,M)/\beta_M$ for $L\in \cM_V$. Note that it
follows from Lemma \ref{lem:multiplicative} that the definition of $\bP$ is
independent of the choice of module $M\in \cM_V$.
If there is an isomorphism $L\oplus E\cong M$ of
$R$-modules, where $L$ and $M$ are in $\cM_V$ and $E$ is finite, and
$\iota$ is the inclusion $L\hookrightarrow L\oplus E\cong M$, then
by definition, we have $\ia(L,M)=\ii\aa(c)$,
where $c$ is the commensurability $c=(L,\id,\iota)\colon L\rightleftharpoons M$.
Thence it immediately follows that $\bP$ satisfies the conclusion of part \ref{item:P},
and it is clearly the unique such probability distribution.

Part \ref{item:equivdistr}, finally, follows from Proposition \ref{prop:plusproj}.
\end{proof}\noindent
%


\section{Class groups of imaginary abelian fields}\label{sec:imagab}\noindent
In the present section, we prove Theorem \ref{thm:introimagabelian},
and use it to give a disproof of Heuristic \ref{he:CLM}. We begin by
establishing some notation for the section and recalling some well-known facts
that will also be useful in the next section.
\subsection*{Generalities on group rings}
Let $\bar{\Q}$ be an algebraic closure of $\Q$. Let $G$ be a finite abelian group,
with dual $\hat{G}=\Hom(G,\bar{\Q}^\times)$. For $\chi$, $\chi'\in \hat{G}$,
we write $\chi\sim\chi'$ if $\ker\chi = \ker\chi'$, or equivalently if there
exists $\sigma \in \Gal(\bar{\Q}/\Q)$ with $\chi=\sigma\circ\chi'$. Each
$\chi\in \hat{G}$ extends to a ring homomorphism $\chi\colon \Q[G]\to \bar{\Q}$,
of which the image is the cyclotomic field $\Q(\chi(G))$, and the natural map
$\Q[G]\to \prod_{\chi\in \hat{G}/\sim}\Q(\chi(G))$ is an isomorphism of $\Q$-algebras.

Let $S$ be a set of prime numbers not dividing $\#G$, and write
$\ZSG=\Z_{(S)}[G]$, which is a maximal $\Z_{(S)}$-order in $\Q[G]$. For $\chi\in \hat{G}$,
the image $\chi(\ZSG)$ of $\ZSG$ in $\Q(\chi(G))$ is the ring $\Z_{(S)}[\chi(G)]$,
which is a Dedekind domain. We have a ring isomorphism
$\ZSG\cong \prod_{\chi\in \hat{G}/\sim} \chi(\ZSG)$, so each $\ZSG$-module $M$ decomposes as
a direct sum $\bigoplus_{\chi\in \hat{G}/\sim}(\chi(\ZSG)\otimes_\ZSG M)$,
which leads to a group isomorphism $\G(\ZSG)\cong \bigoplus_{\chi\in \hat{G}/\sim}\G(\chi(\ZSG))$,
where we recall from the introduction that $\G$ denotes the Grothendieck group.

Since for each $\chi\in \hat{G}$ the ring $\chi(\ZSG)=\Z_{(S)}[\chi(G)]$ is a
Dedekind domain, by \cite[Theorems 1.4.12 and 3.1.13]{Kthry} there is a
canonical isomorphism $\G(\chi(\ZSG))\cong \Cl_{\chi(\ZSG)}\oplus$ $\Z$, where
$\Cl_{\chi(\ZSG)}$ is the class group of the $\Z_{(S)}$-order $\chi(\ZSG)$, and
in particular is finite.
Explicitly, the projection map
$\G(\chi(\ZSG))\to \Z$ is defined by sending the class of a finitely
generated $\chi(\ZSG)$-module $M$ to $\dim_{\Q(\chi(G))}(\Q(\chi(G))\otimes_{\chi(\ZSG)}M)$,
and a canonical splitting $\Z\to \G(\chi(\ZSG))$ is given by $1\mapsto [\chi(\ZSG)]$;
moreover,
if the $\chi(\ZSG)$-module $M$ is a non-zero ideal of $\chi(\ZSG)$, then
the element $[M]$ of $\G(\chi(\ZSG))$ projects to the ideal class of $M$ under
the projection map to $\Cl_{\chi(\ZSG)}$.
In particular, if $M$ is a finite $\ZSG$-module, then the class $[M]$ in $\G(\ZSG)$
is contained in the torsion subgroup $\G(\ZSG)_{\tors}\cong \bigoplus_{\chi\in \hat{G}/\sim}\Cl_{\chi(\ZSG)}$.
%
\subsection*{The Cohen--Lenstra--Martinet prediction}
In the present subsection we show that Heuristic \ref{he:CLM}
implies that if $c\in G$ is an element of order $2$, and
we choose $A=\Q[G]/(1+c)$, $R=\Z_{(S)}[G]/(1+c)$, and $P$ to be
the zero $A$-module in the heuristic,
then the class of
$R\otimes_{\Z[G]}\Cl_F$ is equidistributed in $\G(R)_{\tors}$
as $(F,\iota)$ varies over $\cF$.
The main technical ingredient is Theorem \ref{thm:analytic}.
It is a generalisation of \cite[Corollary 3.7]{CL} from the trivial
character to arbitrary Dirichlet characters.

Let $\cM(\ZSG)$ be a set of finite
${\ZSG}$-modules such that for every finite ${\ZSG}$-module $M$ there is
a unique $M'\in\cM(\ZSG)$ satisfying $M\cong M'$.
For $M\in \cM(\ZSG)$, $\us=(s_\chi)_{\chi\in \ccR}\in \C^{\ccR}$,
and $\uu=(u_\chi)_{\chi\in \ccR}\in (\Z_{\geq 0})^{\ccR}$, we recall the following
definitions from \cite{CL} and \cite{CM}:
\begin{eqnarray*}
|M|^{\us}&=&\prod_{\chi\in \ccR}\#(\chi(\ZSG)\otimes_{\ZSG} M)^{s_{\chi}};\\
\rs_{\uu}(M)&=&\#\{\text{$\ZSG$-linear surjections }Q\to M\},
\end{eqnarray*}
where $Q$ is a projective ${\ZSG}$-module such that
$\dim_{\Q(\chi(G))}(\Q(\chi(G))\otimes_{\ZSG} Q)=u_\chi$ for all $\chi\in \hat{G}$;
it is easy to see that $\rs_{\uu}(M)$ is well-defined, i.e. does not depend on
the choice of $Q$;
\begin{eqnarray*}
w_{\uu}(M)&=&\frac{\rs_{\uu}(M)}{|M|^{\uu}}\cdot \frac{1}{\#\Aut M};\\
w_{\binfty}(M) &=& \lim_{\uu\to \binfty}w_{\uu}(M),
\end{eqnarray*}
where we recall that the notation $\uu\to \binfty$ was defined at the beginning
of Section \ref{sec:CL}. Observe that for $\uu\in (\Z_{\geq 0})^{\ccR}$,
one has $|M|^{\uu}=\#\Hom(Q,M)$, where $Q$ is as in the definition of $\rs_{\uu}$. Moreover, if
for all $\chi\in \hat{G}$ the ranks $u_\chi$ are ``large'', then ``most'' homomorphisms
$Q\to M$, in a precisely quantifiable sense, are surjective. Making this precise,
one deduces that $w_{\binfty}(M)=1/\#\Aut(M)$. If $\uu$ and $\us$ are as above, and
$f\colon \cM(\ZSG)\to \C$ is any function, we define the following quantities when the
respective limit exists:
\begin{eqnarray*}
\Zz_{\uu}(f,\us)&=&\lim_{\uB\to \binfty}\sum_{M\in \cM({\ZSG})_{\leq \uB}}w_{\uu}(M)|M|^{-\us}f(M);\\
\Zz(f,\us)&=&\Zz_{\binfty}(f,\us)=\lim_{\uu\to\binfty}\Zz_{\uu}(f,\us);\\
\Zz_{\uu}(\us)&=& \Zz_{\uu}(\triv,\us),\;\;\;\Zz(\us)= \Zz(\triv,\us),
\end{eqnarray*}
where $\triv$ denotes the function $M\mapsto 1$ for all $M\in \cM({\ZSG})$, and
$\cM({\ZSG})_{\leq \uB}$ is defined analogously to $\cM_{\leq \uB}$ from Section \ref{sec:CL}.
If $\ZSG'$ is a quotient of $\ZSG$, we also define $\Zz_{\uu}^{\ZSG'}(f,\us)$
analogously to $\Zz_{\uu}(f,\us)$, and $\Zz^{\ZSG'}(f,\us)$ analogously to
$\Zz(f,\us)$, but with the sums running only over $\ZSG$-modules that
factor through $\ZSG'$, i.e. that are annihilated by the kernel of the quotient
map $\ZSG\to \ZSG'$, and we again set $\Zz^{\ZSG'}(\us) = \Zz^{\ZSG'}(\triv,\us)$.
With these definitions, the limit in (\ref{eq:expectation})
can be rewritten as
\begin{eqnarray}\label{eq:ZfZ}
\lim_{\us\to \us_P}\frac{\Zz^R(f,\us)}{\Zz^R(\us)},
\end{eqnarray}
where $\us_P=(\dim_{\Q(\chi(G))}(\Q(\chi(G))\otimes_{\ZSG} P))_{\chi\in \ccR}$,
provided that the limit $\lim_{\us\to \us_P}\Zz^R(f,\us)$ exists and is finite and
that $\Zz^R(\us_P)\neq 0$. We stress that this is true even if the infinite sum
$\Zz^R(\us_P)$ diverges, in which case both the limit in (\ref{eq:expectation})
and that in (\ref{eq:ZfZ}) are equal to $0$.

If ${\ZSG'}=\chi(T)$ for some $\chi\in \cR$, then
$\Zz_{\uu}^{\ZSG'}(f,\us)$ as a function of $\us$ depends only on the entry
$s_{\chi}$ of $\us$, and similarly for $\uu$, so we will write $\Zz_{u_\chi}^{\chi(\ZSG)}(f,s_\chi)$
in this case. In particular, if $f$ is multiplicative in direct sums of the form
$M=\bigoplus_{\chi\in \ccR}M_\chi$, where for each $\chi\in \cR$, the summand
$M_\chi$ is a $\chi(T)$-module, then one has
$$
\Zz_{\uu}(f,\us)= \prod_{\chi\in \ccR}\Zz_{u_{\chi}}^{\chi(\ZSG)}(f,s_\chi).
$$

Recall from the discussion at the beginning of the section that for each
$\chi\in \hat{G}$ we have a canonical isomorphism
$\G(\chi(\ZSG))\cong \Cl_{\chi(\ZSG)}\oplus$ $\Z$,
and that if $M$ is a finite $\chi(\ZSG)$-module, then $[M]\in \G(\chi(\ZSG))$ is
contained in the torsion subgroup $\Cl_{\chi(\ZSG)}$ of $\G(\chi(\ZSG))$. 
\begin{theorem}\label{thm:analytic}
Let $\hat{G}'$ be a subset of $\hat{G}$, let $\ZSG'=\prod_{\chi\in \ccRd}\chi(\ZSG)$,
let $\phi=(\phi_\chi)_{\chi\in \ccRd}\colon \bigoplus_{\chi\in \ccRd}\Cl_{\chi(\ZSG)}
\to \C^{\times}$ be a group homomorphism,
and define $f\colon \cM({\ZSG})\to \C$ by $f(M)=0$ if $M$ does not factor through $\ZSG'$,
and $f(M)=\phi([M])$ otherwise. For $\chi\in \cRd$, define $\tau_\chi=0$ if
$\phi_\chi$ is trivial, and $\tau_\chi=-1/[\Q(\chi(G)):\Q]$ otherwise.
Assume that $S$ contains all but finitely many prime numbers.
Then for all $\us\in \C^{\ccR}$ satisfying $\Re s_\chi > \tau_\chi$ for all $\chi\in \cRd$,
we have
$$
\Zz(f,\us)=\Zz^{\ZSG'}(f,\us) =
\prod_{\chi\in \ccRd}\prod_{k=1}^\infty L_{\Q(\chi(G))}^{(S)}(\phi_\chi^{-1},s_{\chi}+k),
$$
where $L_{\Q(\chi(G))}^{(S)}(\phi_\chi^{-1},s)$ denotes the Dirichlet
$L$-function corresponding to the Dirichlet character $\phi_\chi^{-1}$ of the field
$\Q(\chi(G))$ with the Euler factors at prime
ideals not dividing any element of $S$ omitted.
\end{theorem}
\begin{proof}
The proof proceeds very similarly to those of \cite[Corollary 3.7]{CL} and
\cite[Corollaire 3.9(ii)]{CM}, so we only sketch the main steps.

One has, for all $M$, $M_1$, $M_2\in \cM({\ZSG})$ and 
$\uu$, $\uv\in (\Z_{\geq 0})^{\ccR}$, the identities
$$
\rs_{\uu+\uv}(M)=\sum_{N\subset M}\rs_{\uu}(N)\rs_{\uv}(M/N)(\#N)^{\uv}
$$
with the sum running over submodules $N$ of $M$, and
\begin{eqnarray*}
\sum_{M\in \cM({\ZSG})}\!(w_{\uu}(M)\cdot\#\{N\subset M:N\cong M_1,M/N\cong M_2\})=
w_{\uu}(M_1)w_{\uu}(M_2);
\end{eqnarray*}
see \cite[Proposition 3.2 and Theorem 3.5]{CL}. Using the fact that $f$ is
multiplicative in short exact sequences of modules, one deduces from these two
identities and a short calculation that there is a formal identity of Dirichlet
series $\Zz_{\uu+\uv}(f,\us)=\Zz_{\uu}(f,\us)\cdot\Zz_{\uv}(f,\uu+\us)$. In
particular, if $\triv_\chi\in (\Z_{\geq 0})^{\ccR}$ denotes the element that
has $\chi$-entry $1$ and all other entries equal to $0$, then
$$
\Zz_{\uu+\triv_\chi}(f,\us)=\Zz_{\uu}(f,\us)\cdot\Zz_{1}^{\chi(T)}(f,u_\chi+s_\chi).
$$
A direct calculation shows that for each $\chi\in \hat{G}$, one has
$\Zz_{1}^{\chi(\ZSG)}(f,s)=L_{\Q(\chi(G))}^{(S)}(\phi_{\chi}^{-1},s+1)$.
It follows that for all $\uu\in (\Z_{\geq 0})^{\ccR}$,
one has a formal identity between Dirichlet series
\begin{eqnarray}\label{eq:prodL}
\Zz_{\uu}(f,\us) = \prod_{\chi \in \ccRd}\prod_{k=1}^{u_\chi}L_{\Q(\chi(G))}^{(S)}(\phi_{\chi}^{-1},s_\chi+k).
\end{eqnarray}
By \cite[Ch. VIII, Theorems 5 and 7]{Lang}, the Dirichlet series for
$L_{\Q(\chi(G))}^{(S)}(\phi_{\chi}^{-1},s+1)$
converges for $\Re s> \tau_\chi$, and for all $k\in\Z_{\geq 2}$, the Dirichlet series
for $L_{\Q(\chi(G))}^{(S)}(\phi_{\chi}^{-1},s+k)$ converges absolutely for $\Re s> \tau_\chi$.
It follows from this and from a classical result of Landau on
convergence of products of Dirichlet series (see \cite[Theorem 54]{HardyRiesz})
that equation (\ref{eq:prodL}) is an equality of analytic functions, valid whenever
$\Re s_\chi > \tau_\chi$ for all $\chi\in \cRd$.
Finally, since, for every $s\in \C$ with $\Re s> \tau_\chi$, one has
$L_{\Q(\chi(G))}^{(S)}(\phi_{\chi}^{-1},s+k)=1+O(2^{-k})$,
we may take the limit of equation (\ref{eq:prodL}) as $\uu\to \binfty$,
and the result follows.
\end{proof}
Now, let $c\in G$ be an element of order $2$, let $A=\Q[G]/(1+c)$,
let $R=\Z_{(S)}[G]/(1+c)$, let $P$ be the zero $A$-module, and let
$\cM$ and $\cF$ be as in the introduction.
Let $\hat{G}'=\{\chi\in\hat{G}: \chi(c) = -1\}$, so that
$\G(R)_{\tors} \cong \bigoplus_{\chi\in \hat{G}'/\sim}\Cl_{\chi(T)}$.
For a group homomorphism
$\phi=(\phi_{\chi})_{\chi\in \hat{G}'/\sim}\colon \G(R)_{\tors}\to \C^\times$,
define $f_{\phi}\colon \cM\to \C^\times$
by $f_{\phi}(M) = \phi([M])$ for all $M\in \cM$.
\begin{proposition}\label{prop:equidistr}
  Suppose that for all group homomorphisms $\phi\colon \G(R)_{\tors} \to \C^{\times}$,
  Heuristic \ref{he:CLM} holds with $K=\Q$, with $R$ and $P$ as just defined,
  and with $f=f_{\phi}$. Then, as $(F,\iota)$ ranges over $\cF$, the class
  of $R\otimes_{\Z[G]}\Cl_F$ in $\G(R)_{\tors}$ is equidistributed.
\end{proposition}
\begin{proof}
It follows from Theorem \ref{thm:analytic} that if
$\phi=(\phi_{\chi})_{\chi\in \hat{G}'/\sim}$ is a group
homomorphism as in the hypotheses, then the limit
\eqref{eq:expectation} for $f=f_{\phi}$ is equal to
$$
\lim_{s\to 0} \prod_{\chi\in \hat{G}'/\sim}\prod_{k=1}^\infty 
\left(\frac{L^{(S)}_{\Q(\chi(G))}(\phi_{\chi}^{-1},s+k)}{\zeta^{(S)}_{\Q(\chi(G))}(s+k)}\right),
$$
where $\zeta_{\Q(\chi(G))}^{(S)}(s)$ denotes the Dedekind zeta function of
the field $\Q(\chi(G))$ with the Euler factors at prime ideals not dividing
any element of $S$ omitted.
If $\phi$ is non-trivial, i.e. at least one $\phi_{\chi}$ is non-trivial,
then for any such $\chi$ the pole of the Dedekind zeta function of $\Q(\chi(G))$
at $s=1$ ensures that this limit is $0$.
If $x\in \G(R)_{\tors}$ is arbitrary, and $f_x\colon \cM\to \C$ is defined
by $f_x(M)=1$ if $[M]=x$ and $f_x(M)=0$ otherwise, then by the usual
Fourier theory (which, in this case, is character theory of finite abelian groups), we have 
$$
f_x = \frac{1}{\#\G(R)_{\tors}}\sum_{\phi}\overline{\phi(x)}\cdot f_\phi,
$$
where the sum runs over all homomorphisms $\phi\colon \G(R)_{\tors}\to \C^\times$.
Thus, the limit \eqref{eq:expectation} for $f=f_{x}$ is $\frac{1}{\#\G(R)_{\tors}}$,
with the only non-zero contribution coming from the trivial homomorphism, as claimed.
\end{proof}

In the remainder of the section we will show that the conclusion of Proposition
\ref{prop:equidistr} is, in fact, false, in general.
\subsection*{Roots of unity}
Let $F/\Q$ be an imaginary abelian field with Galois group $G$, and let $c\in G$
be the automorphism of $F$ given by complex conjugation.
We write $\hat{G}^-=\{\chi\in \hat{G}: \chi(c)=-1\}$, and $\ZSG^-=\ZSG/(1+c)$;
the latter ring may be identified with $\prod_{\chi\in \hat{G}^-/\sim}\chi(\ZSG)$,
and one has
\begin{eqnarray}\label{eq:grothendieck}
\G(\ZSG^-) \cong \bigoplus_{\chi\in \hat{G}^-/\sim}\G(\chi(\ZSG)).
\end{eqnarray}
For each $m\in \Z_{>0}$, we denote by $\zeta_m$ a primitive $m$-th root
of unity in some algebraic closure of $F$.
Let $U$ be the set of prime numbers $q\in S$ for which $F$ contains $\zeta_q$.
Recall from the introduction that $\mu_F$ denotes the group of roots of unity
in $F$.
\begin{proposition}\label{prop:A}
The group $\ZSG^-\otimes_{\Z[G]}\mu_F$ is cyclic of order $\prod_{q\in U}q$.
\end{proposition}
\begin{proof}
Since $\mu_F$ is finite cyclic, the group $\Z_{(S)}\otimes_{\Z}\mu_F$ is cyclic
of order equal to the largest divisor of $\#\mu_F$ that is a product of primes in
$S$. If $q$ is a prime number for which $q^2$ divides $\#\mu_F$, then $q$ divides
$[F:\Q]=\#G$, and therefore $q\not\in S$. This implies that $\Z_{(S)}\otimes_{\Z}\mu_F$
is cyclic of order $\prod_{q\in U}q$. It is also a $\Z_{(S)}[G]$-module on which
$c$ acts as $-1$, so it equals $\ZSG^-\otimes_{\Z[G]}\mu_F$.
\end{proof}\noindent
For each $q\in U$, denote by $\varphi_q\colon \ZSG^-\to \End\langle\zeta_q\rangle\cong \Z/q\Z$
the ring homomorphism that describes the $\ZSG^-$-module structure of $\langle \zeta_q\rangle$.
\begin{proposition}\label{prop:B}
For each $q\in U$ there is an element $\chi_q\in \hat{G}^-$, unique up to $\sim$, such
that $\varphi_q$ factors as $\ZSG^-\to \chi_q(\ZSG)\to \Z/q\Z$, and it is characterised
by the subfield $F^{\ker\chi_q}\subset F$ being equal to $\Q(\zeta_q)$. Also, if $\fp_q$
denotes the kernel of $\chi_q(\ZSG)\to\Z/q\Z$, then the image of the element
$[\ZSG^-\otimes_{\Z[G]}\mu_F]\in \G(\ZSG^-)$ under the isomorphism {\rm (\ref{eq:grothendieck})}
equals the image of $(\chi_q(\ZSG)/\fp_q)_{q\in U}$ under the natural inclusion
$\bigoplus_{q\in U}\G(\chi_q(\ZSG))\subset \bigoplus_{\chi\in \hat{G}^-/\sim}\G(\chi(\ZSG))$.
\end{proposition}
\begin{proof}
Since $\Z/q\Z$ is a field, the map $\varphi_q$ factors through exactly one of the
components in the decomposition $\ZSG^-=\prod_{\chi\in \hat{G}^-/\sim}\chi(\ZSG)$,
say through $\chi_q(\ZSG)$. By the irreducibility of the $q$-th cyclotomic polynomial,
the induced map $G\to \chi_q(G)\to (\Z/q\Z)^\times$ is surjective, and
since $q$ does not divide $\#G$, the map $\chi_q(G)\to (\Z/q\Z)^\times$ is injective,
so the map $\chi_q(G)\to (\Z/q\Z)^\times$ is an isomorphism. This shows that
$\ker\chi_q=\ker(G\to (\Z/q\Z)^\times)$, so we have indeed $F^{\ker \chi_q}=\Q(\zeta_q)$.
In particular, the order of $\chi_q$ equals $q-1$, so $\chi_q\neq \chi_{q'}$
for distinct $q$, $q'\in U$. We have an isomorphism of $\ZSG^-$-modules $\ZSG^-\otimes_{\Z[G]}\mu_F
\cong \bigoplus_{q\in U}\chi_q(\ZSG)/\fp_q$, and this implies the last assertion.
\end{proof}
\subsection*{Maximal ideals}
Let $\chi \in \hat{G}^-$. We describe the set of non-zero prime ideals of the
Dedekind domain $\chi(\ZSG)=\Z_{(S)}[\chi(G)]$, which for non-empty $S$ coincides
with the set $\Maxspec \chi(\ZSG)$ of its maximal ideals. Denote by $[\chi]\in \hat{G}^-$
the equivalence class of $\chi$ under $\sim$. For each $p\in S$, let an embedding
of $\bar{\Q}$ in an algebraic closure $\bar{\Q}_p$ of the field $\Q_p$ of $p$-adic
numbers be fixed, so that the group $\hat{G}=\Hom(G,\bar{\Q}^\times)$ may be
identified with $\Hom(G,\bar{\Q}^\times_p)$; for $\psi$, $\psi'\in [\chi]$,
we write $\psi\sim_p \psi'$ if there exists $\sigma\in \Gal(\bar{\Q}_p/\Q_p)$
with $\psi'=\sigma\circ\psi$. For each $(p,\psi)\in S\times[\chi]$, the map
$\psi$ induces a ring homomorphism $\psi\colon \chi(\ZSG)\to \bar{\Q}_p$ of which
the image is contained in the discrete valuation ring $\Z_p[\psi(G)]$; we
write $\fm_{p,\psi}$ for the kernel of the resulting map from $\chi(\ZSG)$ to
the residue class field of $\Z_p[\psi(G)]$. This kernel is a non-zero prime ideal
of $\chi(\ZSG)$. Each non-zero prime ideal of $\chi(\ZSG)$ is of the form
$\fm_{p,\psi}$ with $(p,\psi)\in S\times[\chi]$, and one has $\fm_{p,\psi}=\fm_{p',\psi'}$
if and only if $p=p'$ and $\psi\sim_p\psi'$.
\begin{example}\label{ex:Teichmueller}
The prime ideal $\fp_q$ of $\chi_q(\ZSG)$ occurring in Proposition \ref{prop:B}
equals $\fm_{q,\omega_q}$, where $\omega_q\colon G\to \Z_q^\times$ is the unique
group homomorphism for which the induced map $G\to (\Z/q\Z)^\times$ describes the
$G$-module structure of $\langle\zeta_q\rangle$. The character $\omega_q\in \hat{G}$
is called the \emph{Teichm\"uller character at $q$}.
\end{example}
\subsection*{Bernoulli numbers} Let $\chi\in \hat{G}^-$, and let $f(\chi)\in \Z_{>0}$
be minimal with $F^{\ker\chi}\subset \Q(\zeta_{f(\chi)})$. For each $t\in \Z$
that is coprime to $f(\chi)$, denote by $\eta_t$ the restriction to $F^{\ker \chi}$
of the automorphism of $\Q(\zeta_{f(\chi)})$ that sends $\zeta_{f(\chi)}$ to
$\zeta_{f(\chi)}^t$; note that $\eta_t$ belongs to the Galois group
$\Gal(F^{\ker\chi}/\Q)$, which may be identified with $G/\ker\chi$ and with $\chi(G)$,
and indeed we shall view $\eta_t$ as an element of the cyclotomic field $\Q(\chi(G))$.
We define the Bernoulli number
$$
\beta(\chi)=
\sum_{\genfrac{}{}{0pt}{}{1\leq t\leq f(\chi),}{\gcd(t,f(\chi))=1}}\frac{t}{f(\chi)}\cdot\eta_t^{-1},
$$
which is also an element of the cyclotomic field $\Q(\chi(G))$. For each
$(p,\psi)\in S\times[\chi]$, the character $\psi$ induces a field embedding
$\Q(\chi(G))\to\bar{\Q}_p$, which we simply denote by $\psi$. The following
result relates the image $\psi(\beta(\chi))$ of the Bernoulli number to the
$\psi$-component $\Z_p[\psi(G)]\otimes_{\Z[G]}\Cl_F$ of the class group of $F$.

\begin{proposition}\label{prop:C}
\begin{enumerate}[leftmargin=*, label={\upshape(\alph*)}]
\item\label{item:Ca} Let $\chi\in \hat{G}^-$ and $(p,\psi)\in S\times[\chi]$. If $q\in U$ is
such that $\chi\sim\chi_q$, assume $(p,\psi)\neq(q,\omega_q)$, where $\omega_q$ is
as in Example \ref{ex:Teichmueller}. Let $g$ be the
rank of $\Z_p[\psi(G)]$ as a $\Z_p$-module. Then the order of
$\Z_p[\psi(G)]\otimes_{\Z[G]}\Cl_F$ equals a unit of $\Z_p[\psi(G)]$ times $\psi(\beta(\chi))^g$.
\item\label{item:Cb} Let $q\in U$. Then $\Z_q[\omega_q(G)]\otimes_{\Z[G]}\Cl_F=\{0\}$, and
$q\cdot \omega_q(\beta(\chi_q))\in \Z_q^\times$.
\end{enumerate}
\end{proposition}
\begin{proof}
Part \ref{item:Ca} is, stated with different notation, the same as
Theorem 2 in \cite[Ch. 1, \S 10]{MW}. The first assertion of part \ref{item:Cb}
is Remark 1 in \cite[Ch. 1, \S 10]{MW} following the same theorem. For the second
assertion of \ref{item:Cb}, note that by Proposition \ref{prop:B} we have
$F^{\ker \chi_q}=\Q(\zeta_q)$, so $f(\chi_q)=q$, and therefore
$$
q\cdot \omega_q(\beta(\chi))=\sum_{t=1}^{q-1}t\cdot\omega_q(\eta_t)^{-1}.
$$
Here $\omega_q(\eta_t)$ is a $(q-1)$-th root of unity in $\Z_q$, and, by
definition of $\omega_q$ and $\eta_t$, it maps to $(t \bmod q)$ in $\Z/q\Z$. Hence,
$q\cdot\omega_q(\beta(\chi_q))$ is an element of $\Z_q$ that maps to $(q-1 \bmod q)$
and therefore belongs to $\Z_q^\times$.
\end{proof}\noindent
The following result describes, for each $\chi\in \hat{G}^-$, the finite
$\chi(\ZSG)$-module $\chi(\ZSG)\otimes_{\Z[G]}\Cl_F$ up to Jordan--H\"older isomorphism
in terms of Bernoulli numbers. We let $\fp_q$ be the prime ideal of $\chi_q(\ZSG)$
introduced in Proposition \ref{prop:B}.
\begin{proposition}\label{prop:D}
Let $\chi\in \hat{G}^-$.
\begin{enumerate}[leftmargin=*, label={\upshape(\alph*)}]
\item\label{item:Da} If there does not exist $q\in U$ such that $\chi\sim \chi_q$, then one
has $\beta(\chi)\in \chi(\ZSG)$, and the $\chi(\ZSG)$-modules $\chi(\ZSG)\otimes_{\Z[G]}\Cl_F$
and $\chi(\ZSG)/(\beta(\chi))$ have isomorphic Jordan--H\"older series.
\item\label{item:Db} If $q\in U$ is such that $\chi=\chi_q$, then one has $\fp_q\beta(\chi)
\subset \chi(\ZSG)$, and the $\chi(\ZSG)$-modules $\chi(\ZSG)\otimes_{\Z[G]}\Cl_F$
and $\chi(\ZSG)/\fp_q\beta(\chi)$ have isomorphic Jordan--H\"older series.
\end{enumerate}
\end{proposition}
\begin{proof}
In case \ref{item:Da}, one sees from Proposition \ref{prop:C}\ref{item:Ca} that
for every non-zero prime ideal $\fm=\fm_{p,\psi}$ of $\chi(\ZSG)$ the element
$\beta(\chi)\in \Q(\chi(G))$ has non-negative valuation at $\fm$, so one has
$\beta(\chi)\in \chi(\ZSG)$. In case \ref{item:Db}, Proposition \ref{prop:C} shows
that the same assertion has the single exception $\fm=\fp_q=\fm_{q, \omega_q}$,
and that $\beta(\chi)$ has valuation $-1$ at $\fp_q$; so in that case one has
$\fp_q\beta(\chi)\subset \chi(\ZSG)$.

Proposition \ref{prop:C} implies that $\beta(\chi)\neq 0$, so all $\chi(\ZSG)$-modules
occurring in Proposition \ref{prop:D} are finite; and two finite $\chi(\ZSG)$-modules
have isomorphic Jordan--H\"older series if and only if, for each non-zero prime
ideals $\fm=\fm_{p,\psi}$ of $\chi(\ZSG)$, their $\fm$-primary parts have the same
cardinality. In case \ref{item:Da}, Proposition \ref{prop:C}\ref{item:Ca} shows
that this is indeed the case for the two modules $\chi(\ZSG)\otimes_{\Z[G]}\Cl_F$
and $\chi(\ZSG)/(\beta(\chi))$. In case \ref{item:Db}, there is again a single
$\fm=\fm_{q,\omega_q}$ that requires special treatment; Proposition
\ref{prop:C}\ref{item:Cb} shows that in this exceptional case both modules have
trivial $\fm$-primary parts.
\end{proof}\noindent
We are now ready to prove Theorem \ref{thm:introimagabelian}. We first recall
the statement.
\begin{theorem}\label{thm:imagabelian}
Let $F/\Q$ be a finite imaginary abelian extension, let $G$ be its Galois group,
let $c\in G$ denote complex conjugation, let $S$ be a set of prime
numbers not dividing $\#G$, and let $\ZSG^-=\Z_{(S)}[G]/(1+c)$. Then we have
$[\ZSG^-\otimes_{\Z[G]}\Cl_F]=[\ZSG^-\otimes_{\Z[G]}\mu_F]$
in $\G(\ZSG^-)$.
\end{theorem}
\begin{proof}
We will compute $[\ZSG^-\otimes_{\Z[G]}\Cl_F]\in \G(\ZSG^-)\cong\bigoplus_{\chi\in \hat{G}^-/\sim}\G(\chi(\ZSG))$
component by component, and compare the result with Proposition \ref{prop:B}.
Let $\chi\in \hat{G}^-$. If for all $q\in U$ one has $\chi\not\sim\chi_q$, then
by Proposition \ref{prop:D}\ref{item:Da} the $\chi$-component
$[\chi(\ZSG)\otimes_{\Z[G]}\Cl_F]$ equals $[\chi(\ZSG)/(\beta(\chi))]$, and the exact
sequence
$$
0\longrightarrow \chi(\ZSG) \stackrel{\beta(\chi)}{\longrightarrow}\chi(\ZSG)
\longrightarrow \chi(\ZSG)/(\beta(\chi))\longrightarrow 0
$$
shows that $[\chi(\ZSG)/(\beta(\chi))]=0$. If, on the other hand, $q\in U$ is
such that $\chi=\chi_q$, then one has $[\chi(\ZSG)\otimes_{\Z[G]}\Cl_F]=
[\chi(\ZSG)/\fp_q\beta(\chi)]$, and the exact sequence
$$
0\longrightarrow \fp_q \stackrel{\beta(\chi)}{\longrightarrow}\chi(\ZSG)
\longrightarrow \chi(\ZSG)/\fp_q\beta(\chi)\longrightarrow 0
$$
shows that $[\chi(\ZSG)/\fp_q\beta(\chi)]=[\chi(\ZSG)] - [\fp_q] = [\chi(\ZSG)/\fp_q]$.
Comparing this description of $[\ZSG^-\otimes_{\Z[G]}\Cl_F]$ with the description
of $[\ZSG^-\otimes_{\Z[G]}\mu_F]$ given by Proposition \ref{prop:B}, one finds
that they are equal.
\end{proof}\noindent
\begin{corollary}\label{cor:imagabelian}
Under the hypotheses of Theorem \ref{thm:imagabelian}, suppose that for all
$q\in U$, the class group of $\Q(\zeta_{q-1})$ is trivial. Then
$[\ZSG^-\otimes_{\Z[G]}\Cl_F]=0$ in $\G(\ZSG^-)$.
\end{corollary}
\begin{proof}
By Theorem \ref{thm:imagabelian}, we have $[\ZSG^-\otimes_{\Z[G]}\Cl_F]=[\ZSG^-\otimes_{\Z[G]}\mu_F]$
in $\G(\ZSG^-)$. Since for all $q\in U$, the class group of $\Q(\zeta_{q-1})$ is
trivial, the ideals $\fp_q$ from Proposition \ref{prop:B} are principal for all
$q\in U$, so the result follows from Proposition \ref{prop:B}.
\end{proof}
\subsection*{Counterexample to the Cohen--Lenstra--Martinet heuristics}
We will now show that the conclusion of Corollary \ref{cor:imagabelian}
contradicts the conclusion of Proposition \ref{prop:equidistr}, thus disproving
Heuristic \ref{he:CLM}.


Let $G$ be a cyclic group of order $58$, let $S$ contain all prime numbers
except $2$ and $29$, and let $c\in G$ be the unique element of order $2$.
With these choices, if $F/\Q$ is a Galois extension with Galois group isomorphic
to $G$ such that complex conjugation is given by $c$, then
$\ZSG^-\cong \Z_{(S)}\times \Z_{(S)}[\zeta_{29}]$. Let $R=\ZSG^-$, let $P$ be
the zero $A$-module, and let $\cM$ and $\cF$ be as in the introduction.
With these choices, $\cF$ is the family of imaginary cyclic number fields
of degree $58$. The group $\G(R)_{\tors}$ is isomorphic to
the class group $\Cl_{\Z_{(S)}[\zeta_{29}]}$, which one can check to be
elementary abelian of order $8$.
%
%
For every $F$ in the family under consideration except for $F=\Q(\zeta_{59})$,
the set $U$ that was introduced before
Proposition \ref{prop:A} is either empty or equal to $\{3\}$.
It follows from Corollary \ref{cor:imagabelian} that, unless $F=\Q(\zeta_{59})$,
the class $[R\otimes_{\Z[G]}\Cl_F]$ in $\G(R)$ is $0$, contradicting
the conclusion of Proposition \ref{prop:equidistr}.

%
%


\section{Arakelov class groups of real abelian fields}\label{sec:realab}\noindent
In the present section we reinterpret Theorem \ref{thm:imagabelian} in terms
of so-called oriented Arakelov class groups, proving Theorem
\ref{thm:introrealabelian} along the way. This reinterpretation leads to a more
general result, Theorem \ref{thm:allabelian}, which pertains to all finite abelian
extensions of $\Q$, and which might conceivably be true in much greater generality;
see Question \ref{que}.

We will use the generalities on group rings explained at the beginning of
Section \ref{sec:imagab}. In particular, $G$ still denotes a finite abelian group,
and we retain the notation $S$, $\ZSG$, $\hat{G}$.

Assume that $S$ does not contain any prime numbers dividing $2\cdot \#G$.
Let $F$ be a real abelian number field with Galois group $G$ over $\Q$, and
let $\cO$ denote the ring of integers of $F$. If $k\neq \Q$ is a subfield of
$F$ that is cyclic over $\Q$, let $\alpha_k=\Norm_{\Q(\zeta_f)/k}(1-\zeta_f)$,
where $f\in \Z_{>0}$ is minimal with $k\subset \Q(\zeta_{f})$, and $\zeta_f$
denotes an arbitrarily chosen primitive $f$-th root of $1$ in an algebraic
closure of $F$.
Let $D$ be the subgroup of $F^\times$ generated by the $G$-orbits of $\alpha_k$
for all subfields $k\neq \Q$ of $F$ that are cyclic over $\Q$. We follow Mazur--Wiles \cite{MW}
in defining the \emph{group of cyclotomic units} of $F$ to be
$C=D\cap \cO^\times = \ker(|\Norm_{F/\Q}|\colon D\to \Q^\times)$.
This group forms a $\Z[G]$-submodule of $\cO^\times$. We will write $-_{(S)}$
for $\Z_{(S)}\otimes_{\Z}-$.
We will write the $\ZSG$-modules $D_{(S)}$, $C_{(S)}$, $\cO^\times_{(S)}$ additively,
so that $C_{(S)}$ is the kernel of the map $D_{(S)}\to D_{(S)}$, $x\mapsto \Norm_{F/\Q}x$,
where $\Norm_{F/\Q} = \sum_{g\in G}g\in \ZSG$.
%
%

\begin{proposition}\label{prop:cycunits}
The $\ZSG$-module $C_{(S)}$ is isomorphic to $\ZSG/(\sum_{g\in G}g)$.
\end{proposition}
\begin{proof}
We will use without further explicit mention the facts that the ring $\Z_{(S)}$ is
flat over $\Z$, and, for every $\chi\in \hat{G}$, the ring $\chi(\ZSG)$ is flat over $\ZSG$.

By Theorem 1 in \cite[Ch. 1, \S 10]{MW}, the subgroup $C$ has finite index in
$\cO^\times$. It follows that the annihilator of $C_{(S)}$ in $\ZSG$ is generated
by $\sum_{g\in G}g\in \ZSG$. The assertion of the proposition is therefore
equivalent to the statement that for every non-trivial $\chi\in\hat{G}$, the
module $\chi(\ZSG)\otimes_{\ZSG} C_{(S)}$ is free of rank $1$ over $\chi(\ZSG)$, which we
will now prove.

First, suppose that $\chi\in \hat{G}$ is faithful, non-trivial. In particular,
$G$ is non-trivial, cyclic. Then for all $k\subsetneq F$ distinct from $\Q$, we have
$\chi(\ZSG)\otimes_\ZSG \ZSG\alpha_k=0$, since all elements of
$\ZSG\alpha_k$ are fixed by a proper subgroup of $G$, while $\chi$ is faithful.
Since the image of $\Norm_{F/\Q}\in \ZSG$ in the quotient $\chi(\ZSG)$ is $0$, we have
\begin{eqnarray*}
 \chi(\ZSG)\otimes_\ZSG C_{(S)} & = &\chi(\ZSG)\otimes_\ZSG \ker(\Norm_{F/\Q}\colon D_{(S)} \to D_{(S)})\\
&=& \chi(\ZSG)\otimes_\ZSG D_{(S)} = \chi(\ZSG)\otimes_\ZSG \ZSG\alpha_F,
\end{eqnarray*}
so $\chi(\ZSG)\otimes_\ZSG C_{(S)}$ is cyclic.

Now, we deduce the general case. Let $\chi\in \hat{G}$ be arbitrary
non-trivial, and let $F'\subset F$ be the fixed field of $\ker \chi$.
Temporarily write $C_{F,(S)}=C_{(S)}$, and let $C_{F',(S)}=\Z_{(S)}\otimes_{\Z}C_{F'}$,
where $C_{F'}$ is the analogously defined group of cyclotomic units of $F'$.
The image of the element $\Norm_{F/F'}=\sum_{g\in \ker \chi}g$ of $\ZSG$ in the
quotient $\chi(\ZSG)$ is $[F:F']$, which is invertible in $\chi(\ZSG)$. It follows that
$\chi(\ZSG)\otimes_\ZSG C_{F,(S)}=\chi(\ZSG)\otimes_\ZSG\Norm_{F/F'}(C_{F,(S)})$. Moreover, it
follows from a direct calculation, which we leave to the reader, that
$\Norm_{F/F'}C_{F,(S)}\subset C_{F',(S)}$,
and in fact, we have equality, since $\Norm_{F/F'}$ acts as $[F:F']$, and thus
invertibly, on $C_{F',(S)}\subset C_{F,(S)}$. In summary, we have
$\chi(\ZSG)\otimes_\ZSG C_{F,(S)}= \chi(\ZSG)\otimes_\ZSG C_{F',(S)}$. The assertion therefore
follows from the special case proved above, applied to $F'$ in place of $F$.
\end{proof}\noindent
The map $G\to G$ given by $g\mapsto g^{-1}$ for all $g\in G$ extends
$\Z_{(S)}$-linearly to an
isomorphism between $\ZSG$ and its opposite ring $\ZSG^{\opp}$. If $M$ is any $\ZSG$-module,
then the above isomorphism makes $\Hom(M,\Z_{(S)})$ and $\Hom(M,\Q/\Z)$ into
$\ZSG$-modules, which we denote by $M^*$ and $M^\lor$, respectively.

We can now prove Theorem \ref{thm:introrealabelian} from the introduction.
Let us recall the statement.
\begin{theorem}\label{thm:realabelian}
Let $F/\Q$ be a finite real abelian extension, let $G$ be its Galois group, let $S$ be
a set of prime numbers not dividing $2\cdot \#G$, and let $\ZSG=\Z_{(S)}[G]$.
Then we have the equality $[\ZSG\otimes_{\Z[G]}\Ar_F] = [\ZSG]-[\Z_{(S)}]$ in $\G(\ZSG)$,
where $[\Z_{(S)}]$ denotes the class of $\Z_{(S)}$ with the trivial $G$-action.
\end{theorem}
\begin{proof}
Recall from the exact sequence (\ref{eq:DualArakelov}) that we have the equality
$$
[\ZSG\otimes_{\Z[G]} \Ar_F] = [(\cO_{(S)}^\times)^*]
+[(\ZSG\otimes_{\Z[G]}\Cl_F)^\lor]
$$ in $\G(\ZSG)$. By Theorem 1 in \cite[Ch. 1, \S 10]{MW}, we
have $[\cO_{(S)}^\times/C_{(S)}]= [\ZSG\otimes_{\Z[G]} \Cl_F]$, or dually,
$$
[(C_{(S)})^*/(\cO_{(S)}^\times)^*]=
[(\ZSG\otimes_{\Z[G]}\Cl_F)^\lor],
$$
whence $[\ZSG\otimes_{\Z[G]} \Ar_F] = [(C_{(S)})^*]$. The theorem therefore
follows from Proposition \ref{prop:cycunits}.
\end{proof}\noindent
Theorem \ref{thm:realabelian} expresses that the class of the
$\ZSG$-module $\ZSG\otimes_{\Z[G]}\Ar_F$ in $\G(\ZSG)$ is
``as simple as it can be'' given the $\Q[G]$-module structure of
$\Q[G]\otimes_{\Z[G]}\Ar_F$. Indeed, as the discussion at the beginning of
Section \ref{sec:imagab} shows, the class $[\ZSG\otimes_{\Z[G]}\Ar_F]$ in $\G(\ZSG)$ is determined
by its images under the projection $\G(\ZSG)\to \G(\ZSG)/\G(\ZSG)_{\tors}$ and under the
canonical splitting $\G(\ZSG)\to \G(\ZSG)_{\tors}$; the former is determined
by the $\Q[G]$-module structure of $\Q[G]\otimes_{\Z[G]}\Ar_F$, and the theorem
implies that the latter is $0$.

Theorems \ref{thm:realabelian} and \ref{thm:imagabelian} can be elegantly combined
into one statement, using the so-called oriented Arakelov class group of
a number field $F$, defined by Schoof in \cite{Schoof}. To explain this, we
will briefly recall the definition and the most salient properties of the
oriented Arakelov class group, and refer the reader to \cite{Schoof} for details.

Let $F/\Q$ be a finite extension, let $\Id_F$ be the group of fractional ideals of $\cO_F$,
and let $F_{\R}$ denote the $\R$-algebra $\R\otimes_{\Q}F$. The maximal compact
subgroup $\comp(F_{\R}^\times)$ of $F_{\R}^\times$ is isomorphic to
$\{ \pm 1 \}^{\fr}\times \{z\in \C^{\times} : |z|=1\}^{\fc}$, where $\fr$ and
$\fc$ denote the set of real, respectively complex places of $F$. It contains
the group $\mu_F$ of roots of unity of $F$ as a subgroup. We have canonical maps
$\Id_F\to \R_{>0}$ and $F_{\R}^\times\to \R_{>0}$, the first given by the ideal
norm, and the second given by the absolute value of the $\R$-algebra norm. Let
$\Id_F\times_{\R_{>0}} F_{\R}^\times$ be the fibre product with respect to these
maps. The \emph{oriented Arakelov class group} $\widetilde{\Pic}^0_F$ of $F$ is
defined as the cokernel of the map $F^\times\to \Id_F\times_{\R_{>0}} F_{\R}^\times$
that sends $\alpha\in F^\times$ to $(\alpha\cO_F,\alpha)$. It is a compact abelian
$\Aut_F$-module, whose dual $\Hom(\widetilde{\Pic}^0_F,\R/\Z)$ will be denoted by
$\widetilde{\Ar}_F$. One has an exact sequence of finitely generated discrete $\Aut_F$-modules
$$
0 \to \Ar_F \to \widetilde{\Ar}_F \to \Hom(\comp(F_{\R}^\times)/\mu_F,\R/\Z) \to 0.
$$
Let $K$ be an algebraic number field, let $F/K$ be a finite Galois extension,
let $G$ be its Galois group, let $S$ is any set of odd prime numbers, let
$\Sigma$ be the set of infinite places of $K$, and for $v\in \Sigma$
let $I_v\subset G$ denote an inertia subgroup at $v$.
Then one has an isomorphism of $\Z_{(S)}[G]$-modules
$$
\Z_{(S)}\otimes_{\Z}\Hom(\comp(F_{\R}^\times),\R/\Z) \cong \bigoplus_{v\in \Sigma}\Ind_{G/I_v}\tau_v,
$$
where $\Ind_{G/I_v}$ denotes the induction from $I_v$ to $G$, and
$$
\tau_v = \leftchoice{\Z_{(S)}[I_v]/\langle \sum_{g\in I_v}g\rangle}{v\text{ is real;}}
{\Z_{(S)}}{v\text{ is complex}.}
$$
Let $d$ be the degree of $K$ over $\Q$. By combining the above observation with
the exact sequence (\ref{eq:DualArakelov}), we deduce the equalities
$$
[\widetilde{\Ar}_{F,(S)}] - [\Ar_{F,(S)}] = d\cdot [\Z_{(S)}[G]] - \sum_{v\in \Sigma}[\Z_{(S)}[G/I_v]] - [\mu_{F,(S)}^\lor],
$$
and hence
\begin{eqnarray}\label{eq:ArTilde}
\lefteqn{[\widetilde{\Ar}_{F,(S)}] =}\\
& & d\cdot [\Z_{(S)}[G]] - \sum_{v\in \Sigma}[\Z_{(S)}[G/I_v]]+
[(\cO^\times_{F,(S)})^*]+[\Cl_{F,(S)}^\lor] - [\mu_{F,(S)}^\lor]\nonumber
\end{eqnarray}
in $\G(\Z_{(S)}[G])$, where $\Z_{(S)}[G/I_v]$ denotes the permutation module
with a $\Z_{(S)}$-basis given by the set of cosets $G/I_v$, and with $G$ acting
by left multiplication. Note that this equality
holds for \emph{all} finite Galois extensions, not just abelian ones.
From this equality, together with Theorems \ref{thm:realabelian} and
\ref{thm:imagabelian}, one easily deduces the following result.
\begin{theorem}\label{thm:allabelian}
Let $F/\Q$ be a finite abelian extension, let $G$ be its Galois group, let $S$
be a set of prime numbers not dividing $2\cdot \#G$, and let $\ZSG=\Z_{(S)}[G]$.
Then we have the equality
$[\ZSG\otimes_{\Z[G]} \widetilde{\Ar}_F] = [\ZSG]-[\Z_{(S)}]$ in $\G(\ZSG)$.
\end{theorem}\noindent
The theorem suggests the following question.
\begin{question}\label{que}
Let $K$ be an algebraic number field, let $d$ be its degree over $\Q$, let
$F/K$ be a finite Galois extension, let $G$ be its Galois group, let
$S$ be a set of prime numbers not dividing $2\cdot \#G$, and let $\ZSG=\Z_{(S)}[G]$.
Is the class of $\widetilde{\Ar}_{F,(S)}$ in $\G(\ZSG)$ equal to $d\cdot[\ZSG]-[\Z_{(S)}]$?
\end{question}\noindent
%
%
Equation (\ref{eq:ArTilde}) shows that the answer to Question \ref{que} is
affirmative when the natural map $\G(\ZSG) \to \G(\Q\otimes_{\Z_{(S)}}\ZSG)$ is
injective, which is the case for example when $S$ is finite, as can be deduced
from Lemma \ref{lem:uniqueproj} and \cite[Theorem (21.4)]{MaxOrders}.


\section{Enumerating number fields}\label{sec:quartics}\noindent
In this section, we give our second disproof of the Cohen--Lenstra--Martinet heuristics.
We begin by proving Theorem \ref{thm:quartics} from the introduction,
and then compare its consequences with the predictions of Heuristic \ref{he:CLM}.

Our disproof suggests that the discriminant is not a good invariant to use for
purposes of arithmetic statistics, and we investigate alternatives.

Let $\bar{\Q}$ be a fixed algebraic closure of $\Q$, and let $\cC(x)$ be the
set of all fields $F\subset \bar{\Q}$ that are Galois over $\Q$ with cyclic Galois group
of order $4$ and whose discriminant is at most $x$. If $\quadsub\subset \bar{\Q}$ is a
quadratic field, let $\cC_{\quadsub}(x)=\{F\in \cC(x): \quadsub\subset F\}$.
If $n$ is a positive integer, let $\numdiv(n)$ denote the number of positive divisors of $n$.
The following result is a more precise version of Theorem \ref{thm:quartics} from
the introduction.
\begin{theorem}\label{thm:proportionquad}
Define
$$
\tempfac=\frac{24+\sqrt{2}}{24}\cdot\prod_{p\equiv 1\!\!\!\!\pmod 4}\left(1+\frac{2}{(p+1)\sqrt{p}}\right)-1,
$$
where $p$ ranges over primes.
If $\quadsub\subset \bar{\Q}$ is a quadratic field, and $d$ is its discriminant,
then define $p_{\quadsub}$ by
\begin{enumerate}[leftmargin=*,label={\upshape(\roman*)}]
  \item\label{item:pkzero} $p_{\quadsub}=0$ if $d<0$ or $d$ has at least one prime divisor that is congruent to
    $3\pmod 4$,
  \item $p_{\quadsub} = \frac{\numdiv(d)}{\tempfac\cdot\prod_{p\mid d}(p+1)\sqrt{p}}$
    if $d>0$, all odd prime divisors of $d$ are congruent to $1\pmod{4}$, and $d$ is even,
    where the product runs over all prime divisors of $d$,
  \item $p_{\quadsub}=\frac{\numdiv(d)}{16\tempfac\cdot\prod_{p\mid d}(p+1)\sqrt{p}}$
    if $d>0$ and all prime divisors of $d$ are congruent to $1\pmod{4}$,
    where the product runs over all prime divisors of $d$.
\end{enumerate}
Then:
\begin{enumerate}[leftmargin=*,label={\upshape(\alph*)}]
  \item\label{item:emptyquart} in case \ref{item:pkzero} the set
    $\cC_{\quadsub}(x)$ is empty for all real numbers $x$;
\item\label{item:subfieldprop} in the other two cases, the limit
$$
\lim_{x\to \infty} \frac{\#\cC_{\quadsub}(x)}{\#\cC(x)}
$$
exists, and is equal to $p_{\quadsub}$;
\item\label{item:sumpk} one has $\sum_{\quadsub}p_{\quadsub}=1$, where the sum
  runs over all quadratic fields in $\bar{\Q}$.
\end{enumerate}
\end{theorem}
\begin{proof}
First we prove part \ref{item:emptyquart}. If $\quadsub$ is a quadratic field, and
$F$ is a cyclic quartic field
containing $\quadsub$, then any place of $\Q$ that ramifies in $\quadsub$ must
be totally ramified in $F$, so $\quadsub$ must be real, and only primes that are congruent
to $1$ or $2\pmod 4$ can ramify in $\quadsub$. This proves the first assertion.

Now we prove part \ref{item:subfieldprop}. Let $k$ be a quadratic field, let
$d$ be its discriminant, and assume that $d$ is positive and not
divisible by any primes that are congruent to $3\pmod 4$. We will use the
estimates of \cite{OuWilliams}. Write $d=2^\beta d'$, where $\beta\in \{0,3\}$,
and $d'$ is a product of distinct prime numbers that are congruent to $1\pmod 4$.
Then the discriminant of any cyclic quartic field containing $k$ is of the form
$n=2^\alpha d'^3a^2$, where $a$ is an odd square-free positive integer that is
coprime to $d'$, and $\alpha=11$ if $\beta=3$, and $\alpha\in \{0,4,6\}$ if
$\beta=0$ (see \cite{OuWilliams}). Note that $k=\Q(\sqrt{n})$, so the
discriminant of a cyclic quartic field determines its unique quadratic subfield.
Let $h(n)$ be the number of cyclic quartic fields inside $\bar{\Q}$
of discriminant $n$. By \cite[equation (3.3)]{OuWilliams} we have
$$
h(n)=\left\{\!\!\begin{array}{ll}
2\numdiv(d'):\!\!\! & \alpha=11,\\
\numdiv(d'):\!\!\! & \alpha=6,\\
\frac{1}{2}\numdiv(d'):\!\!\! & \alpha=0\text{ or }4.
\end{array}\right.
$$
For a real number $x$, let $\sqsum(x)$ denote the number of
square-free positive integers that are at most $x$ and coprime to $2d$.
It follows from the above discussion, that if $d$ is odd, then
\begin{eqnarray*}
\#\cC_{\quadsub}(x)
=\!\! \sum_{a\leq \sqrt{x/d^3}}\!\!h(d^3a^2)
+\!\!\! \sum_{a\leq \sqrt{x/(2^4d^3)}}\!\!\!h(2^4d^3a^2)+
\sum_{a\leq \sqrt{x/(2^6d^3)}}\!\!\!h(2^6d^3a^2)\\
= \numdiv(d)\left(\tfrac{1}{2}\sqsum\bigl(\sqrt{x/d^3}\bigr)+
\tfrac{1}{2}\sqsum\bigl(\sqrt{x/(2^4d^3)}\bigr)+
\sqsum\bigl(\sqrt{x/(2^6d^3)}\bigr)\right),\nonumber
\end{eqnarray*}
where the sums run over square-free positive integers $a$ that are coprime to $2d$;
while if $d$ is even, then
\begin{eqnarray*}
\#\cC_{\quadsub}(x)
= \sum_{a\leq \sqrt{x/(2^{11}d'^3)}}\!\!\!h(2^{11}d'^3a^2)
= 2\numdiv(d')\sqsum\bigl(\sqrt{x/(2^{11}d'^3)}\bigr),
\end{eqnarray*}
with the sum again running over square-free positive integers $a$ that are coprime
to $2d$. A standard estimate shows that for a positive integer $m$, we have
$$
\sqsum(x)=\frac{6x}{\pi^2}\prod_{p|2d}\frac{p}{p+1}+O(x^{1/2}),
$$
where $p$ ranges over primes. Combining this
estimate with the above formulae for $\#\cC_k(x)$ yields
\begin{eqnarray*}
\#\cC_{\quadsub}(x) & = & \numdiv(d)\frac{3x^{1/2}}{d^{3/2}\pi^2}\prod_{p|d}\frac{p}{p+1}+O(x^{1/4})\\
& = & \numdiv(d)\frac{3x^{1/2}}{\pi^2}\prod_{p|d}\frac{1}{(p+1)\sqrt{p}}+O(x^{1/4})
\end{eqnarray*}
if $d$ is odd, and
\begin{eqnarray*}
\#\cC_{\quadsub}(x) & = & \numdiv(d')\frac{x^{1/2}}{2^{5/2}d'^{3/2}\pi^2}\prod_{p|d'}\frac{p}{p+1}+O(x^{1/4})\\
& = & \numdiv(d')\frac{x^{1/2}}{2^{5/2}\pi^2}\prod_{p|d'}\frac{1}{(p+1)\sqrt{p}}+O(x^{1/4})
\end{eqnarray*}
if $d$ is even, where again $p$ ranges over primes.
On the other hand, by \cite{OuWilliams} we have
$$
\#\cC(x) = \frac{3}{\pi^2}tx^{1/2}+O(x^{1/3}(\log x)^3),
$$
whence part \ref{item:subfieldprop} of the theorem follows.

The last part is readily verified by an easy direct computation, and we leave
this to the reader.
\end{proof}\noindent
\begin{theorem}\label{thm:quarticsbody}
Let $\cC(x)$ be the
set of cyclic quartic fields inside $\bar{\Q}$ with discriminant at most $x$,
and let $\cC'(x)\subset \cC(x)$ be the subset of those for
which the class number of the quadratic subfield is not divisible by $3$. Then
the limit $\lim_{x\to \infty}\#\cC'(x)\big/\#\cC(x)$ exists, and one has
$$
\lim_{x\to \infty}\#\cC'(x)\big/\#\cC(x)\approx 0.9914.
$$
\end{theorem}
\begin{proof}
For a quadratic field $\quadsub$ and a positive real number $x$, let
$p_{\quadsub}(x)=\#\cC_{\quadsub}(x)/\#\cC(x)$. Then by Theorem \ref{thm:proportionquad}
one has $\lim_{x\to \infty}p_{\quadsub}(x)=p_{\quadsub}$, where $p_{\quadsub}$
is defined as in that theorem.

Let $\cK$ be any set of real quadratic fields. Then we claim that
\begin{eqnarray}\label{eq:Fatou}
\lim_{x\to \infty}\sum_{\quadsub\in \cK}p_{\quadsub}(x)=\sum_{\quadsub\in \cK}\lim_{x\to \infty}p_{\quadsub}(x).
\end{eqnarray}
This is clear when $\cK$ is finite. In general, for every
finite $\cK'\subset \cK$ we have
$$
\liminf_{x\to \infty}\sum_{\quadsub\in \cK}p_{\quadsub}(x)\geq
\liminf_{x\to \infty}\sum_{\quadsub\in \cK'}p_{\quadsub}(x),
$$
so applying identity (\ref{eq:Fatou}) to finite subsets of $\cK$ we obtain
\begin{eqnarray}\label{eq:liminf}
\liminf_{x\to \infty}\sum_{\quadsub\in \cK}p_{\quadsub}(x)\geq
\sup_{\genfrac{}{}{0pt}{}{\cK'\subset \cK}{\text{finite}}}
\sum_{\quadsub\in \cK'}\lim_{x\to \infty}p_{\quadsub}(x)=
\sum_{\quadsub\in \cK}\lim_{x\to \infty}p_{\quadsub}(x),
\end{eqnarray}
where the last equality follows from the fact that all the summands are non-negative.
By Theorem \ref{thm:proportionquad}\ref{item:sumpk} one has
$\sum_{\quadsub}\lim_{x\to \infty}p_k(x)=1$, where the sum runs over all real
quadratic fields. By combining this with inequality (\ref{eq:liminf}) applied to
the complement of $\cK$ in the set of all real quadratic fields in place of $\cK$,
one deduces that $\limsup_{x\to \infty}\sum_{\quadsub\in \cK}p_{\quadsub}(x) \leq
\sum_{\quadsub\in \cK}\lim_{x\to \infty}p_{\quadsub}(x)$, whence the claimed
equality (\ref{eq:Fatou}) follows.

When applied to the set $\cK$ of real quadratic fields whose class number is
not divisible by $3$, this shows that the limit $\lim_{x\to \infty}\#\cC'(x)\big/\#\cC(x)$
exists. Moreover, by summing $p_{\quadsub}$ from Theorem
\ref{thm:proportionquad} over those real
quadratic fields of discriminant less than $3.1\times 10^9$ whose class number is,
respectively is not, divisible by $3$ one obtains sufficiently tight upper and
lower bounds on that limit to obtain the estimate in the theorem.
\end{proof}\noindent
\begin{remark}\label{rmrk:cycquartic}
If, in Heuristic \ref{he:CLM}, we take
$G=\langle g | g^4=\id\rangle$, $A=\Q[G]/(1+g)\cong \Q$, $S=\{3\}$ (so that in
particular $R\cong\Z_{(3)}$), the module $P$ to be free of rank $1$ over $R$, and
$$
f\colon M\mapsto \leftchoice{0}{3\mid \#M;}{1}{3\nmid \#M,}
$$
then the value of
(\ref{eq:average}) is nothing but the limit $\lim_{x\to \infty}\#\cC'(x)/\cC(x)$
referred to in Theorem \ref{thm:quarticsbody}. The same argument
as in the proof of the theorem shows that, more generally,
if one chooses $A$ and $G$ as just described, $S$ to be any set of odd primes,
the projective module $P$ to have rank $1$ over $R$, and $f$ to be any computable
bounded $\C$-valued function on $\cM$, then the limit (\ref{eq:average}) exists
and can be computed to any desired precision.
\end{remark}\noindent
Theorem \ref{thm:quarticsbody} contradicts the predictions of Heuristic \ref{he:CLM},
as we will now explain. As just remarked, the limit $\lim_{x\to \infty}\#\cC'(x)/\cC(x)$
of Theorem \ref{thm:quarticsbody} is equal to the value of (\ref{eq:average})
for suitable choices of $G$, $A$, $S$, $P$, and $f$. The value of (\ref{eq:expectation}), on
the other hand, with these choices is the same as with a
different choice, namely as with $G'=\langle h|h^2=\id\rangle$,
$A'=\Q[G']/(1+h)\cong A$, and $S'=S$, $P'=P$, and $f'=f$. With that latter
choice, the value of (\ref{eq:expectation}) was computed in \cite[(1.2)(b)]{CM2}
to be $\approx 0.8402$, and thus not equal to the value of (\ref{eq:average}),
which is given by Theorem \ref{thm:quarticsbody}. This completes our disproof
of Heuristic \ref{he:CLM}.

The above disproof relies on the observation that
enumerating number fields by non-decreasing discriminant has the undesirable
feature that certain fields may appear with positive probability as intermediate
fields, so that Heuristic \ref{he:CLM} for those number fields clashes with
that for the intermediate fields. We conjecture that instead enumerating fields
by the ideal
norm $c_{F/K}$ of the product of the primes of $\cO_K$ that ramify in $F/K$ does
not exhibit this feature. The following is a
result in support of this conjecture. Recall that if $\cF$ is a set of pairs
$(F,\iota)$ as in the introduction, where all fields $F$ are Galois extensions
of a given field $K$, then for every positive real number $B$ we define
$\cF_{c\leq B}=\{(F,\iota)\in \cF: c_{F/K}\leq B\}$.

\begin{proposition}
Let $K$ be a number field, let $\bar{K}$ be an algebraic closure of $K$, let $G$ be
a finite abelian group, let $V$ be a finitely generated $\Q[G]$-module, and let $\cF$
be the set of all pairs $(F,\iota)$, where $F\subset \bar{K}$ is a Galois
extension of $K$, and $\iota$ is an isomorphism between the Galois group of
$F/K$ and $G$ that induces an isomorphism $\Q\otimes_{\Z}\cO_F^\times\cong V$
of $\Q[G]$-modules. Assume that $\cF$ is not empty. Let $k\subset \bar{K}$
be a field properly containing $K$. Then the limit
$\lim_{B\to\infty}\#\{(F,\iota)\in \cF_{c\leq B}:k\subset F\}/\#\cF_{c\leq B}$
is zero.
\end{proposition}
\begin{proof}
Let $T$ be an infinite set of prime ideals of $\cO_K$ with odd
residue characteristics that are not totally split in $k/K$. Then for any
$F\subset \bar{K}$ that contains $k$, and for all $\fp \in T$, we have
$F\otimes_K K_{\fp}\not\cong K_{\fp}^{\#G}$ as $K_{\fp}$-algebras, where
$K_{\fp}$ denotes the
field of fractions of the completion of $\cO_K$ at $\fp$. It follows that
for all positive real numbers $B$, we have
$$
\#\{(F,\iota)\in \cF_{c\leq B}:k\subset F\} \leq
\#\{(F,\iota)\in \cF_{c\leq B}: \forall\fp\in T\colon F\otimes_K \!K_{\fp}\not\cong K_{\fp}^{\#G}\}.
$$
By \cite[Theorem 2.1]{MelanieFair}, for every finite subset
$T'$ of $T$, each of the following limits exists, and there is an equality
\begin{eqnarray*}
\lefteqn{\lim_{B\to\infty}\frac{\#\{(F,\iota)\in \cF_{c\leq B}: \forall\fp\in T'\colon F\otimes_K \!K_{\fp}\not\cong K_{\fp}^{\#G}\}}{\#\cF_{c\leq B}}=}\\
& & \prod_{\fp\in T'}\left(\lim_{B\to \infty}
\frac{\#\{(F,\iota)\in \cF_{c\leq B}: F\otimes_K \!K_{\fp}\not\cong K_{\fp}^{\#G}\}}{\#\cF_{c\leq B}}\right);
\end{eqnarray*}
moreover, \cite[Theorem 2.1]{MelanieFair} also implies that each of the factors on the right hand side is bounded
away from $1$ uniformly for all $\fp\in T$. Thus, we have
\begin{eqnarray*}
\lefteqn{\limsup_{B\to\infty}\frac{\#\{(F,\iota)\in \cF_{c\leq B}:k\subset F\}}{\#\cF_{c\leq B}}
\leq}\\
& & \prod_{\fp\in T}\left(\lim_{B\to \infty}
\frac{\#\{(F,\iota)\in \cF_{c\leq B}: F\otimes_K \!K_{\fp}\not\cong K_{\fp}^{\#G}\}}{\#\cF_{c\leq B}}\right)=0,
\end{eqnarray*}
as claimed.
\end{proof}


\section{Reasonable functions}\label{sec:conj}\noindent
In the present section we address the question of which functions may
qualify as ``reasonable'' for the purposes of Conjecture \ref{conj:ourconj}.
We begin by proving Theorem \ref{thm:introBjorn}, which suggests that demanding
that $\bE(|f|)$ should exist is likely not a sufficient criterion.
After that, we offer two possible interpretations of the word
``reasonable''.

From now on, let $(X,p)$ be an infinite discrete probability space such that
$p(x)>0$ for all $x\in X$, where we recall from the introduction
that $p(x)$ is shorthand for $p(\{x\})$, and let $Y$ be the probability space
$X^{\Z_{\geq 1}}$ with the induced product probability measure; see e.g. \cite[\S 38]{Halmos}.
When referring to subsets of a measure space, we will say ``almost all''
to mean a subset whose complement has measure $0$.

Let us reformulate Theorem \ref{thm:introBjorn}, using the above notation.
\begin{theorem}\label{thm:Bjorn}
For almost all sequences $y=(y_i)_i\in Y$,
there exists a function $f\colon X\rar\R_{\geq 0}$ whose expected value $\bE(f)$
is finite, but for which the average $\lim_{n\to\infty}\frac1n\sum_{i=1}^nf(y_i)$
of $f$ on $y$ does not exist in $\R$.
\end{theorem}\noindent
The idea of the proof will be to show that, for a typical $(y_i)_i\in Y$,
there are many elements $x\in X$ that occur much earlier in $(y_i)_i$ than
one would expect. The function $f$ then gives those elements a large weight.
\begin{proposition}\label{prop:manyx}
For almost all sequences $y=(y_i)_i\in Y$ it is true that for all
$\epsilon\in \R_{>0}$ there exist infinitely many $x\in X$ such that for some
$i\leq \epsilon/p(x)$, one has $y_i=x$.
\end{proposition}
\begin{proof}
Let $\epsilon\in \R_{>0}$ be given, and let $U$ be a finite subset of $X$.
First, we claim that for almost all sequences $y=(y_i)_i\in Y$ there exists
$x\in X\setminus U$ such that for some $i\leq \epsilon/p(x)$, one has
$y_i=x$. If $x$ is an element of $X$, let $E_x$ be the set of $y=(y_i)_i\in Y$
for which $y_i\neq x$ for all $i\leq \epsilon/p(x)$. The probability of the
event $E_x$, meaning the measure of $E_x\subset Y$, is equal to
$(1-p(x))^{\lfloor \epsilon/p(x)\rfloor}$, which tends to $e^{-\epsilon}$, as
$p(x)$ tends to $0$, and in particular is uniformly bounded away from $1$ for
all but finitely many $x\in X$. If $x_1$, $x_2$, \ldots, $x_k$ are distinct
elements of $X$, then the events $E_{x_i}$ are generally not independent, but the
probability of $E_{x_k}$ given that all of $E_{x_1}$, \ldots, $E_{x_{k-1}}$
occur is clearly less than or equal to the probability of $E_{x_k}$. It
follows that the probability that $E_x$ occurs for all $x\in X\setminus U$ is at
most $\prod_{x\in X\setminus U}(1-p(x))^{\lfloor \epsilon/p(x)\rfloor} = 0$. This
proves the claim.

Since the number of finite subsets $U$ of $X$ is countable, it follows from
countable subadditivity that if $\epsilon\in \R_{>0}$ is given, then
for almost all $y=(y_i)_i\in Y$ it is true that for
all finite subsets $U$ of $X$ there exists $x\in X\setminus U$ such that for some
$i\leq \epsilon/p(x)$, one has $y_i=x$. This implies that if $\epsilon\in \R_{>0}$
is given, then for almost all $y=(y_i)_i\in Y$ there exist infinitely many $x\in X$
such that for some $i\leq \epsilon/p(x)$, one has $y_i=x$. By applying this
conclusion to countably infinitely many $\epsilon_n\in\R_{>0}$ in place of
$\epsilon$, where $(\epsilon_n)_{n\in \Z_{\geq 1}}$ is a sequence converging to $0$,
and by invoking countable subadditivity again, we deduce the proposition.
\end{proof}\noindent
%
%
We are now ready to prove Theorem \ref{thm:Bjorn}.
\begin{proof}[Proof of Theorem \ref{thm:Bjorn}]
Let $y=(y_i)_i\in Y$ be a sequence for which the conclusion of Proposition \ref{prop:manyx}
holds. Then there is a sequence $x_1$, $x_2$, $\ldots$ of distinct elements of
$X$ such that for each $n\in \Z_{\geq 1}$, we have $\min\{i\in \Z_{\geq 1}: y_i=x_n\} \leq
n^{-3}/p(x_{n})$.

For $n\in \Z_{\geq 1}$, let $i(n)=\min\{j\in \Z_{\geq 1}: y_j=x_n\}
\leq n^{-3}/p(x_{n})$. For $x\in X$, define $f(x)=0$ if $x\neq x_n$ for any $n\in \Z_{\geq 1}$,
and $f(x_n)=n^{-2}/p(x_n)$. Then we have $\bE(f)=\sum_{n\in \Z_{\geq 1}}n^{-2}$,
which converges. On the other hand, for every $n\in \Z_{\geq 1}$, one has
$$
\frac{1}{i(n)}\sum_{j=1}^{i(n)}f(y_j)\geq\frac{f(x_n)}{i(n)}\geq \frac{n^{-2}/p(x_n)}{n^{-3}/p(x_{n})}=n,
$$
which gets arbitrarily large, as $n$ varies, so the limit
$\lim_{n\to \infty}\frac{1}{n}\sum_{j=1}^nf(y_j)$ does not exist.
\end{proof}\noindent
Let us now discuss two possible interpretations of the word ``reasonable'' in
Conjecture \ref{conj:ourconj}. As will hopefully become clear, this section should be treated
as an invitation to the reader to join in our speculations.
If $\cM_V$ is as in Conjecture \ref{conj:ourconj},
let $f\colon \cM_V\to \C$ be a function. For a positive integer $j$, the $j$-th
moment of $f$ is defined to be $\sum_{M\in \cM_V}(f(M)^j\bP(M))$ if the sum
converges absolutely.

\begin{conjecture}[Supplement 1 to Conjecture \ref{conj:ourconj}]\label{conj:suppl1}
If $f\colon \cM_V\to \C$ is a function such that for all $j\in \Z_{\geq 1}$,
the $j$-th moment of $|f|$ exists in $\R$, then $f$ is ``reasonable'' for the purposes
of Conjecture \ref{conj:ourconj}.
\end{conjecture}\noindent
All bounded functions on $\cM_V$, and many unbounded functions of arithmetic
interest satisfy this condition.
This applies to all examples in \cite{CM2}, with the exception of the
functions $f(M)=\#M$, and $f(M)=(\#M)^2$, which often have expected value $\infty$.
Note that in \cite{CM2} the set $S$ of prime numbers was not assumed to
be finite. Here, when we talk about the examples in \cite{CM2}, we mean the
analogues in our setting of the functions considered there.
On the other hand, it can be shown that if $X=\cM_V$, $A\neq 0$, and $S$ is
non-empty, so that $\cM_V$ is infinite, then for the function constructed in
the proof of Theorem \ref{thm:Bjorn} the second moment does not exist.

Let us call a class $\Funcclass$ of $\C$-valued functions on $X$ \emph{promising}
if for all $f\in \Funcclass$, the expected value $\bE(f)$ exists, and for
almost all sequences $y=(y_i)_i\in Y$ it is true that
for all $f\in \Funcclass$ we have $\lim_{n\to\infty}\frac1n\sum_{i=1}^nf(y_i)=\bE(f)$.
By the law of large numbers, for any function $f\colon X\to \C$ for which
$\bE(f)$ exists, the set $\Funcclass = \{f\}$ is promising. It immediately follows
that any countable set $\Funcclass$ of such functions is promising. On the other
hand, Theorem \ref{thm:Bjorn} implies that the class of all functions $f\colon X\to \C$
for which $\bE(f)$ exists is not promising.
An affirmative answer to the following question would strengthen our confidence
in Conjecture \ref{conj:suppl1}.
\begin{question}\label{qn:moments}
Is the class of all functions $f\colon \cM_V\to \C$ for which for all $j\in \Z_{\geq 1}$,
the $j$-th moment of $|f|$ exists in $\R$ promising?
\end{question}\noindent
We have the following weak result in this direction.
\begin{proposition}
The class of all bounded $\C$-valued functions on $X$ is promising.
\end{proposition}
\begin{proof}
For any $c\in \R_{>0}$, let $\Funcclass_c$ be the class of all functions
$f\colon X\to \C$ for which $\sup_{x\in X}|f(x)|\leq c$. We will
first show that for all $c\in\R_{>0}$, the class $\Funcclass_c$ is
promising.

Let $X_1\subset X_2\subset \ldots$ be a sequence of finite subsets of $X$ such
that $\lim_{j\to \infty}\sum_{x\in X_j}p(x)=1$. Fix $c\in \R_{>0}$,
and let $\epsilon\in \R_{>0}$ be arbitrary. Then we may choose
$X_{j(\epsilon)}$ such that $\sum_{x\in X\setminus X_{j(\epsilon)}}p(x)\leq \epsilon/c$.
Moreover, by the strong law of large numbers, there exists a subset
$Y_c(\epsilon)\subset Y$ of measure $1$ with the following property:
for all $y=(y_i)_i\in Y_c(\epsilon)$
there is an $N_y(\epsilon)\in \Z_{\geq 1}$ such that for all $x\in X_{j(\epsilon)}$
and for all $n\geq N_y(\epsilon)$ one has
$|\frac{1}{n}\cdot\#\{i\leq n:y_i=x\}-p(x)| \leq \epsilon/(c\cdot\#X_{j(\epsilon)})$.
It follows that for all $y\in Y_c(\epsilon)$, for all $f\in \Funcclass_c$, and
for all $n\geq N_y(\epsilon)$ one has
\begin{eqnarray*}
\lefteqn{\left|\frac{1}{n}\sum_{i=1}^nf(y_i)-\bE(f)\right| =
\left|\sum_{x\in X}\left(\tfrac{1}{n}\cdot\#\{i\leq n:y_i=x\}-p(x)\right)\cdot f(x)\right|}\\
& \leq &
\sum_{x\in X_{j(\epsilon)}}\left|\left(\tfrac{1}{n}\cdot\#\{i\leq n:y_i=x\}-p(x)\right)\cdot f(x)\right|\\
& & +\sum_{x\in X\setminus X_{j(\epsilon)}}\left|\left(\tfrac{1}{n}\cdot\#\{i\leq n:y_i=x\}-p(x)\right)\cdot f(x)\right|\\
& \leq & \epsilon + \sum_{x\in X\setminus X_{j(\epsilon)}}\tfrac{1}{n}\cdot\#\{i\leq n:y_i=x\}\cdot|f(x)|
+ \sum_{x\in X\setminus X_{j(\epsilon)}}p(x)|f(x)|\\
& \leq & \epsilon + \Big(1-\sum_{x\in X_{j(\epsilon)}}\tfrac{1}{n}\cdot\#\{i\leq n:y_i=x\}\Big)\cdot c
+ \epsilon\\
& \leq & 2\epsilon + \Big(1-\sum_{x\in X_{j(\epsilon)}}\big(p(x)-\frac{\epsilon}{c\cdot\#X_{j(\epsilon)}}\big)\Big)\cdot c\\
& \leq & 4\epsilon.
\end{eqnarray*}
Therefore, if $(\epsilon_n)_{n\in \Z_{\geq 1}}$ is a sequence of positive real numbers converging to
$0$, then the intersection $Y_c=\bigcap_{n\in \Z_{\geq 1}} Y_c(\epsilon_n)$ has measure $1$ and
has the property that for every $y=(y_i)_i\in Y_c$ and for every
$f\in \Funcclass_c$ one has $\lim_{n\to \infty}\tfrac{1}{n}\sum_{i=1}^nf(y_i)=\bE(f)$.

If $(c_n)_{n\in\Z_{\geq 1}}$ is a sequence of positive real numbers tending to $\infty$, then
the class of all bounded functions is equal to $\bigcup_{n\in \Z_{\geq 1}} \Funcclass_{c_n}$,
and the intersection $\bigcap_{n\in \Z_{\geq 1}} Y_{c_n}$ has measure $1$ and has the property
that for every sequence $y=(y_i)_i$ in this intersection and for every bounded
function $f$ one has $\lim_{n\to \infty}\tfrac{1}{n}\sum_{i=1}^nf(y_i)=\bE(f)$.
This completes the proof.
\end{proof}\noindent
We will now describe a completely different approach to the question of
reasonableness, which is based on the idea that one can distinguish between
``highly artificial'' functions, such
as those that are constructed in the proof of Theorem \ref{thm:Bjorn}, and
``natural'' functions that one cares about in practice by the ease with which
they can be computed.

Let $A$, $S$, $R$, and $V$ be as in Conjecture \ref{conj:ourconj}, and suppose
that $S$ is non-empty and $A\neq 0$. Let $Q$ be the
set of non-increasing sequences $(n_i\in \Z_{\geq 0}: i \in \Z_{\geq 0})$
that have only finitely many non-zero terms. Let $\Maxspec(\Zz(R))$
be the finite set of maximal ideals of the centre $\Zz(R)$ of $R$. It follows
from \cite[Lemme 2.7]{CM} that the set $\cM$ is canonically in bijection
with the set $Q^{\Maxspec(\Zz(R))}$, and by Lemma \ref{lem:uniqueproj}
we also have a bijection between $\cM_V$ and $Q^{\Maxspec(\Zz(R))}$.

\begin{conjecture}[Supplement 2 to Conjecture \ref{conj:ourconj}]\label{conj:suppl2}
If $f\colon \cM_V\to \Z$ is a function such that $\bE(|f|)$ exists in $\R$, and
the function
$Q^{\Maxspec(\Zz(R))}\to~\Z$ induced by $f$ is computable in polynomial time, where
the input is given in unary notation, then $f$ is
``reasonable'' for the purposes of Conjecture \ref{conj:ourconj}.
\end{conjecture}\noindent
The functions that one typically cares about in practice, including all those
given as examples in \cite{CM2} for which $S$ is finite,
are computable in polynomial time. On the other hand, we do not expect
the function that was constructed in the proof of Theorem \ref{thm:Bjorn} for
$y$ being the sequence of class groups in a family of number fields to be computable
in polynomial time. Indeed, to define $f(M)$, one needs to know roughly the first
$\#\Aut(M)$ terms of the sequence $y$, and $\#\Aut(M)$ is exponential in
the size of the input.

\end{document}